\documentclass[12pt,reqno,tbtags]{amsart}

\usepackage{amsmath,amssymb,amsthm,bbm,fullpage} %,hyperref}
\usepackage[square,numbers]{natbib}
\bibpunct[, ]{[}{]}{;}{n}{,}{,}
\usepackage{amssymb}
\usepackage{url}
\usepackage{enumerate}
\usepackage{graphicx}

\numberwithin{equation}{section}  % add section numbers to equations

\title{Near-critical SIR epidemic on a random graph with given degrees}

\date{30 March, 2015; revised 21  December, 2015}

\author{Svante Janson}
  \thanks{SJ  supported by the Knut and Alice Wallenberg Foundation}
\author{Malwina Luczak}\thanks{ML supported by an EPSRC Leadership Fellowship EP/J004022/2}
\author{Peter Windridge}
\author{Thomas House}\thanks{TH supported by EPSRC} % EP/J002437/2

\newcommand\urladdrx[1]{{\urladdr{\def~{{\tiny$\sim$}}#1}}}

\address{Department of Mathematics, Uppsala University, PO Box 480,
SE-751~06 Uppsala, Sweden}
\email{svante.janson@math.uu.se}
\urladdrx{http://www.math.uu.se/svante-janson}

\address{School of Mathematical Sciences, Queen Mary University of London, Mile End Road, London, E1 4NS, UK.}
\email{m.luczak@qmul.ac.uk}
\urladdrx{http://www.maths.qmul.ac.uk/~luczak/}

\address{School of Mathematical Sciences, Queen Mary University of London, Mile End Road, London, E1 4NS, UK.}
\email{pete@windridge.org.uk}
\urladdr{http://peter.windridge.org.uk/}

\address{School of Mathematics, University of Manchester, Manchester, M13
  9PL, UK}
\email{thomas.house@manchester.ac.uk}
\urladdr{http://personalpages.manchester.ac.uk/staff/thomas.house/}

\renewcommand\le{\leqslant}
\renewcommand\ge{\geqslant}

\allowdisplaybreaks

\newtheorem{theorem}{Theorem}[section]
\newtheorem{lemma}[theorem]{Lemma}

\newtheorem{corollary}[theorem]{Corollary}

\theoremstyle{definition}

\newtheorem{remark}[theorem]{Remark}

\theoremstyle{remark}

\newenvironment{romenumerate}[1][-10pt]{% optional argument changes indentation
\addtolength{\leftmargini}{#1}\begin{enumerate}% gives (i), (ii) etc.
 \renewcommand{\labelenumi}{\textup{(\roman{enumi})}}%
 \renewcommand{\theenumi}{\textup{(\roman{enumi})}}%
 }{\end{enumerate}}

\newcounter{oldenumi}
{\setcounter{oldenumi}{\value{enumi}}
\begin{romenumerate} \setcounter{enumi}{\value{oldenumi}}}
{\end{romenumerate}}

\newcounter{thmenumerate}

\newcounter{xenumerate}   %no left indentation; thus wider lines

\newcommand\pfitemx[1]{\par#1:}
\newcommand\pfitemref[1]{\pfitemx{\ref{#1}}}

\newcommand{\refT}[1]{Theorem~\ref{#1}}

\newcommand{\refL}[1]{Lemma~\ref{#1}}
\newcommand{\refR}[1]{Remark~\ref{#1}}
\newcommand{\refS}[1]{Section~\ref{#1}}

\newcommand{\refApp}[1]{Appendix~\ref{#1}}
\begingroup
  \divide\count255 by 60
  \multiply\count255 by -60
  \advance\count255 by \time
  \ifnum \count255 < 10 \xdef\klockan{\the\count1.0\the\count255}
  \else\xdef\klockan{\the\count1.\the\count255}\fi
\endgroup

\newcommand{\sumk}{\sum_{k=0}^\infty}

\newcommand{\sumki}{\sum_{k=1}^\infty}

\newcommand\set[1]{\ensuremath{\{#1\}}}

\newcommand\xpar[1]{(#1)}
\newcommand\bigpar[1]{\bigl(#1\bigr)}
\newcommand\Bigpar[1]{\Bigl(#1\Bigr)}
\newcommand\biggpar[1]{\biggl(#1\biggr)}
\newcommand\lrpar[1]{\left(#1\right)}

\newcommand\Bigcpar[1]{\Bigl\{#1\Bigr\}}

\newcommand\abs[1]{|#1|}
\newcommand\bigabs[1]{\bigl|#1\bigr|}
\newcommand\Bigabs[1]{\Bigl|#1\Bigr|}
\newcommand\biggabs[1]{\biggl|#1\biggr|}
\newcommand\lrabs[1]{\left|#1\right|}
\def\rompar(#1){\textup(#1\textup)}    % usage: \rompar(...)
\newcommand\xfrac[2]{#1/#2}

\newcommand\xqfrac[2]{#1/(#2)}

\newcommand\Bigparfrac[2]{\Bigpar{\frac{#1}{#2}}}

\def\xexp(#1){e^{#1}}

\newcommand\floor[1]{\lfloor#1\rfloor}

\newcommand\nn{[n]}
\newcommand\ntoo{\ensuremath{{n\to\infty}}}

\newcommand\punkt[1]{\if.#1\else.\spacefactor1000\fi{#1}}
\newcommand\iid{i.i.d\punkt}
\newcommand\ie{i.e\punkt}
\newcommand\eg{e.g\punkt}

\newcommand\cf{cf\punkt}
\newcommand{\as}{a.s\punkt}

\newcommand\whp{w.h.p\punkt}

\newcommand{\tend}{\longrightarrow}
\newcommand\dto{\overset{\mathrm{d}}{\tend}}
\newcommand\pto{\overset{\mathrm{p}}{\tend}}

\newcommand\eqd{\overset{\mathrm{d}}{=}}

\newcommand\op{o_{\mathrm p}}
\newcommand\Op{O_{\mathrm p}}

\newcommand\bbR{\mathbb R}

\newcounter{CC}
\newcounter{cc}
\newcommand{\cc}{\stepcounter{cc}\ccx} %new constant c_i
\newcommand{\ccx}{c_{\arabic{cc}}}     %repeats the last c_i
\newcommand{\ccdef}[1]{\xdef#1{\ccx}}     %defines #1 as the last c_i
\newcommand\E{\operatorname{\mathbb E{}}}
\renewcommand\P{\operatorname{\mathbb P{}}}
\newcommand\Var{\operatorname{Var}}

\newcommand\Exp{\operatorname{Exp}}
\newcommand\Po{\operatorname{Po}}

\newcommand\Bin{\operatorname{Bin}}

\newcommand\ga{\alpha}
\newcommand\gb{\beta}
\newcommand\gd{\delta}
\newcommand\gD{\Delta}

\newcommand\gam{\gamma}

\newcommand\gl{\lambda}

\newcommand\go{\omega}

\newcommand\gs{\sigma}
\newcommand\gss{\sigma^2}
\newcommand\eps{\varepsilon}

\newcommand\cC{\mathcal C}

\newcommand\cE{\mathcal E}
\newcommand\cF{\mathcal F}

\newcommand\cL{{\mathcal L}}

\newcommand\cR{{\mathcal R}}

\newcommand\cZ{{\mathcal Z}}

\newcommand\qw{^{-1}}
\newcommand\qww{^{-2}}
\newcommand\qq{^{1/2}}
\newcommand\qqw{^{-1/2}}
\newcommand\qqq{^{1/3}}

\renewcommand{\=}{:=}

\newcommand\intoi{\int_0^1}
\newcommand\intoo{\int_0^\infty}

\newcommand\dd{\,\mathrm{d}}

\newcommand\rhs{right-hand side}

\newcommand\gnp{\ensuremath{G(n,p)}}
\newcommand\gnm{\ensuremath{G(n,m)}}

\newcommand\fhe{free half-edge}
\newcommand\fihe{free infective half-edge}

\newcommand{\Prob}{\mathbb{P}}

\newcommand{\N}{\mathbb{N}}

\newcommand{\Indi}[1]{\mathop{\mathbbm{1}_{#1}}\nolimits}

\newcommand{\iEnd}{T^*}

\newcommand{\fXSn}{h_{\mathrm{S},n}}
\newcommand{\fXIn}{h_{\mathrm{I},n}}
\newcommand{\fXRn}{h_{\mathrm{R},n}}

\newcommand\nS{n_{\mathrm{S}}}
\newcommand\nI{n_{\mathrm{I}}}
\newcommand\nR{n_{\mathrm{R}}}

\newcommand{\nk}{n_{k}}
\newcommand\nSk{n_{\mathrm{S},k}}
\newcommand\nIk{n_{\mathrm{I},k}}
\newcommand\nRk{n_{\mathrm{R},k}}
\newcommand\nSx[1]{n_{\mathrm{S},#1}}
\newcommand\nIx[1]{n_{\mathrm{I},#1}}

\newcommand{\X}[1]{X_{#1}}
\newcommand{\XS}[1]{X_{\mathrm{S},#1}}
\newcommand{\XI}[1]{X_{\mathrm{I},#1}}
\newcommand{\XR}[1]{X_{\mathrm{R},#1}}

\newcommand{\Sv}[1]{S_{#1}}
\newcommand{\Iv}[1]{I_{#1}}
\newcommand{\Rv}[1]{R_{#1}}

\newcommand{\Svk}[1]{S_{#1}(k)}
\newcommand{\Svkm}[1]{\tilde{S}_{#1}(k)}

\newcommand{\Rzero}{\mathcal{R}_0}
\newcommand{\DSn}{D_{\mathrm{S},n}}
\newcommand{\DSl}{D_{\mathrm{S}}}
\newcommand{\hDS}{\hat D_{\mathrm{S}}}

\newcommand{\dSi}{d_{\mathrm{S},i}}
\newcommand{\dSmax}{d_{\mathrm{S},*}}
\newcommand{\dImax}{d_{\mathrm{I},*}}
\newcommand{\dIi}{d_{\mathrm{I},i}}

\newcommand{\FIt}{F_{i,t}}

\newcommand\Gr{G}
\newcommand\GrCM{G^*}

\newcommand\xI{x_{\mathrm{I}}}
\newcommand\xR{x_{\mathrm{R}}}

\renewcommand\nn{^{(n)}}

\newcommand\Zzs{Z_{0,*}}

\newcommand\fs{\cZ}
\newcommand\xl{\varkappa}
\newcommand\xxl{(1 + \sqrt{1+2\nu\gliii})/\gliii}
\newcommand\gan{\ga_n}
\newcommand\tgan{\tilde\ga_n}
\newcommand\Zx[1]{\bar Z_{#1}}
\newcommand\Zm{\Zx{m}}
\newcommand\Zmi{\Zx{m+1}}
\newcommand\glii{\gl_2}
\newcommand\gliii{\gl_3}
\newcommand\mxx{m_{**}}
\newcommand\mx{m_*}
\newcommand\deps{\delta}
\newcommand\bZ{\Zx{}}
\newcommand\dxx{\overline{d}}
\newcommand\hW{\widehat W}

\newcommand\cenab{\cE_n^{a,b}}
\newcommand\tand{\text{ and }}
\newcommand\kk{\chi}
\newcommand\Skorohod{Skorohod}
\begin{document}

\keywords{SIR epidemic, random graph with given degrees, configuration model, critical window.}
\subjclass[2010]{05C80, 60F99, 60J28, 92D30}

\begin{abstract}
Emergence of new diseases and elimination of existing diseases is a key public health issue. In mathematical models of epidemics, such phenomena involve the process of infections and recoveries passing through a critical threshold where the basic reproductive ratio is 1.

In this paper, we study near-critical behaviour in the context of a susceptible-infective-recovered (SIR) epidemic on a random
(multi)graph on $n$ vertices with a given degree sequence. We concentrate on
the regime just above the threshold for the emergence of a large epidemic,
where the basic reproductive ratio is $1 + \omega (n) n^{-1/3}$, with
$\omega (n)$ tending to infinity slowly as the population size, $n$, tends
to infinity. We determine the probability  that a large epidemic occurs, and
the size of a large epidemic.

Our results require basic regularity conditions on the degree sequences, and
the assumption that the third moment of the degree of a random susceptible
vertex stays uniformly bounded as $n \to \infty$. As a
corollary, we determine the probability and size of a large
near-critical epidemic on a standard binomial random graph in the `sparse' regime, where the average degree is constant.
As a further consequence of our method, we
obtain an improved result on the size of the giant component in a random
graph with given degrees just above the
critical window, proving a conjecture by Janson and Luczak.
\end{abstract}

\maketitle

\section{Introduction}
\label{s:intro}

Infectious diseases continue to pose a serious threat to individual and public health. Accordingly, health organisations are constantly seeking to analyse and assess events that may present new challenges. These may include acts of bioterrorism, and other events indicating emergence of new infections, which threaten to spread rapidly across the globe facilitated by the efficiency of modern transportation. Likewise, a lot of effort is being directed into suppressing outbreaks of established diseases such as influenza and measles, as well as into eliminating certain endemic diseases, such as polio and rabies.

In an SIR epidemic model, an infectious disease spreads through a population where each individual is either susceptible, infective or recovered. The population is represented by a network (graph) of contacts,
where the vertices of the network correspond to individuals and the edges correspond to potential infectious contacts. Different individuals will have different lifestyles and patterns of activity, leading to different numbers of contacts; for simplicity, we assume that each person's contacts are randomly chosen from among the rest of the population. The degree of a vertex is the number of contacts of the corresponding individual.

We assume that infectious individuals become recovered at rate $\rho \ge 0$ and infect each neighbour at rate $\beta > 0$.
Then the basic reproductive ratio $\cR_0$ (i.e.\ the average number of secondary cases of infection arising from a single case) is given by the average size-biased susceptible degree times the probability that a given infectious contact takes place before the infective individual recovers.

Emergence and elimination of a disease involves the process of infectious transitions and recoveries being pushed across a critical threshold, usually corresponding to the basic reproductive ratio $\cR_0$ equal to 1, see~\cite{ARKB03,BD03,RD13,S09}.
For example, a pathogen mutation can increase the transmission rate and make
a previously `subcritical' disease (i.e.\ not infectious enough to cause a
large outbreak) into a `supercritical' one, where a large outbreak may
occur, see~\cite{ARKB03}. Moreover, after a major outbreak in
the supercritical case, disease in the surviving population is
subcritical. However, subsequently, as people die and new individuals are
born (i.e.\ immunity wanes), $\cR_0$ will slowly increase, and, when it
passes 1,
another major outbreak may occur. Equally, efforts at disease
control may result in subcriticality for a time, but then inattention may
lead to an unnoticed parameter shift to supercriticality.
Thus, under certain conditions, one can expect most large outbreaks to occur close to criticality, and so
there is practical interest in theoretical understanding of the behaviour of near-critical epidemics.
%the population will reach the critical level where a fraction susceptible equals  1/R_0. Just after this the
%community is just above criticality even if R_0>>1.

Critical SIR epidemics have been studied for populations with complete mixing, under different assumptions, by~\cite{ben2004size,GMMG,HJL,M};
%The classical SIR epidemic (on a complete graph) near criticality is
%discussed in \cite{ben2004size}, using non-rigorous arguments.
this is equivalent to studying epidemic processes on the complete graph, or on the Erd\H{o}s-R\'enyi graph $G(n,p)$.
In \cite{ben2004size}, near-criticality is discussed using non-rigorous arguments. Martin-L\"of~\cite{M} studies a generalized Reed-Frost epidemic model, where the number of individuals that a given infective person infects has an essentially arbitrary distribution. The binomial case is equivalent to studying the random graph $G(n,p)$ on $n$ vertices with edge probability $p$.
%or too the complete graph with some probability p for each possible infection.
The author considers the regime where $\cR_0 - 1 = a n^{-1/3}$ and the initial number of infectives is $b n^{1/3}$, for constant $a,b$. A limit distribution is derived for the final size of the epidemic, observing bimodality for certain values of $a$ and $b$ (corresponding to `small' and `large' epidemics). Further analytical properties of the limit distribution are derived in~\cite{HJL}. In~\cite{GMMG}, a standard SIR epidemic for populations with homogeneous mixing is studied, with vaccinations during the epidemic; a diffusion limit is derived for the final size of a near-critical epidemic.

In the present paper,
we address near-critical phenomena in the context of an epidemic spreading in a population of a large size $n$, where
%of disease spreading around a finite population,
%where each individual is either susceptible, infective or recovered.
%Individuals are represented by vertices in a graph (network), with edges corresponding to
%potentially infectious contacts.  Infective vertices become recovered at rate $\rho \ge 0$ and
%infect each neighbour at rate $\beta > 0$.
%those are the only possible transitions, so, in particular, recovered vertices never become
%infective.
the underlying graph (network) is
%In this paper, we consider the SIR epidemic on
a random (multi)graph with given vertex degrees. In other words, we specify the number of contacts for each individual, and consider a graph chosen uniformly at random from among all graphs with the specified sequence of contact numbers.
This random graph model allows for greater inhomogeneity, with a rather arbitrary distribution of the number of contacts for different persons.
We study the regime just above the
critical threshold for the emergence of a large epidemic, where the basic
reproductive ratio is $1 + \omega (n) n^{-1/3}$, with $\omega (n)$ growing large as the population size $n$ grows. (For example, when the population size is about 1 million, we could consider $\cR_0$ of order about $1.01$.)
%tending to infinity slowly as the population size, $n$, tends to infinity.

From the theory of branching processes, at the start of an epidemic, each
infective individual leads to a large outbreak with probability of the order
$\cR_0-1$.
Roughly, our results confirm the following, intuitively clear from the above observation,
picture. If the size $n$ of the population is very large, with
the initial total infectious degree $\XI0$ (i.e.\ total number of potential infectious contacts at the beginning of the epidemic or total number of acquaintances of initially infectious individuals) much larger than
$(\cR_0-1)^{-1}$, then a large epidemic will occur with high probability.
%at least as fast as $n (\cR_0-1)^2$, then a large epidemic will occur with
%high probability; the %same is true if the initial total infectious degree
%grows more slowly than $n (\cR_0-1)^2$ but
%If the initial total infectious degree grows more slowly than $(\cR_0-1)^{-1}$,
%(Since $(\cR_0-1 -1)n^{1/3} = \omega (n) \to \infty$, $(\cR_0-1)^{-1}$ is
%much smaller than $n
%(\cR_0-1)^2$ for large $n$.)
%then the outbreak will be contained with positive probability. Furthermore,
If the initial total infectious degree is much smaller than
$(\cR_0-1)^{-1}$, then the outbreak will be contained with high
probability. In the intermediate case where $\XI0$  and $(\cR_0-1)^{-1}$ are of
the same order of magnitude,
a large epidemic can occur with positive probability, of the order
$\exp \bigpar{-c \XI0 (\cR_0-1)}$, for some positive constant $c$.
So, if the population size is about a million, and $\cR_0$ about $1.01$,
then $\XI0$ much larger than 100 will result in a large epidemic with high
probability. On the other hand, if $\XI0$ is less than 10, say, than the
outbreak will be contained with high probability.

Furthermore, we determine the likely size of a large epidemic. Here, there are three possible regimes, depending on the size of the initial total infectious degree relative to $n (\cR_0-1)^2$. Broadly speaking, if $\XI0$ is much larger than $n (\cR_0-1)^2$, then the total number of people infected will be proportional to $(n \XI0)^{1/2}$. On the other hand, if $\XI0$ is much smaller than $n (\cR_0-1)^2$ then, in the event that there is a large epidemic, the total number of people infected will be proportional to $n (\cR_0-1)$. The intermediate case where $\XI0$ and $n (\cR_0-1)^2$ are of the same order `connects' the two extremal cases.

Note that, if $\XI0$ is of the same or larger order of magnitude than $n (\cR_0-1)^2$ (the first and third case in the paragraph above), then $\XI0 (\cR_0-1)$ is very large, so a large epidemic does occur with high probability. This follows since, by our assumption, $n(\cR_0-1)^3 = \omega (n)^3$ is large for large $n$.

The above results are proven under fairly mild regularity assumptions on the
shape of the degree distribution. We allow a non-negligible proportion of
the population to be initally recovered, i.e.\ immune to the disease.
(This also allows for the possibility that a part, not necessarily random,
of the population is vaccinated before the outbreak, since the vaccinated
individuals can be regarded as recovered.)
We
require that the third moment of the susceptible degree be bounded; in
particular, that implies that the maximum susceptible degree in a population
of size $n$ is of the order no larger than $n^{1/3}$. So in particular, in a
population of size 1 million, the super-spreaders (i.e.\ individuals with
largest numbers of contacts) should not be able to infect more than around
100 individuals.

To demonstrate this behaviour for a particular example, we used stochastic
simulations that make use of special Monte Carlo techniques that allow us to
consider multiple initial conditions within the same realisation of the
process. The algorithm is described in Appendix~\ref{app:sellke}.
Figure~\ref{fig:fs} shows our results for the relationship between the epidemic
final size $\mathcal{Z}$ and the initial force of infection $X_{I,0}$ for 20
realisations of the process, with each realisation involving multiple different
initial conditions, for population sizes $n=10^5$, $n=10^6$ and $n=10^7$.  The
model rate parameters are $\rho=1$ and $\beta=1$; and the network has Poisson
degree distribution with mean $\lambda=2.02$ meaning that $\mathcal{R}_0 =
1.01$.
(The network was generated as an Erd\H{o}s-R\'enyi random graph
with edge probability $\lambda/n$.)
These plots show the emergence of the three epidemic sizes that our
results predict as $n$ increases, i.e.\ `small' epidemics of size $O(1)$,
`large' epidemics of size proportional to $(nX_{I,0})^{1/2}$, and `large'
epidemics of size comparable to $n(\cR_0-1)$.

Epidemics on graphs with given degrees have been considered in a number of
recent studies, both within the mathematical biology and probability
communities. A set of ordinary differential equations approximating the time
evolution of a large epidemic were obtained by Volz~\cite{Volz_2007}, see
also Miller~\cite{miller11}, and  also Miller, Slim, and
Volz~\cite{MSV}. These papers consider the case where the epidemic starts
very small. Differential equations for an epidemic starting with a large
number of infectives appear in \citet{miller14}.
Convergence of the random process to these equations in the case where the
second moment of the degree of a random vertex is uniformly bounded (both
starting with only few infectives and with a large number of
infectives) was proven in \citet{JansonLuczakWindridgeSIR}. (See
also \citet{DDMT12} and \citet{BohmanPicollelli}, where related results are
proven in the case where the fifth moment of the degree of a random vertex
is uniformly bounded and in the case of bounded vertex degrees
respectively. See also~\cite{BR} for results in the case of bounded vertex
degrees and general infection time distributions.)

However, we appear to be the first to study the `barely
supercritical' SIR epidemic on a random graph with given degrees.
As a
corollary, we determine the probability and size of a large
near-critical epidemic on a sparse binomial (Erd\H{o}s-R\'enyi) random graph, also to our knowledge the first such results in the literature.

Our approach also enables us to prove the conjecture of Janson and Luczak in~\cite{JansonLuczakGiantCpt}, establishing their Theorem 2.4 concerning the size of the largest component in the barely supercritical random graph with given vertex degrees under weakened assumptions.
%See Kang and Seierstad \cite{KangSeierstad} for an earlier related result under stronger
%assumptions.

We proceed in the spirit of~\cite{JansonLuczakWindridgeSIR} and
\cite{JansonLuczakGiantCpt}, evolving the epidemic process simultaneously with constructing the random multigraph. The main technical difficulties involve delicate concentration of measure estimates for quantities of interest, such as the current total degrees of susceptible, recovered and infective vertices. Also, our proofs involve couplings of the evolution of the total infective degree with suitable Brownian motions.

The remainder of the paper is organised as follows.
In Section~\ref{s:notationresults}, we define our notation and state our
main results
(Theorems \ref{t:criticalSIR} and \ref{t:takeoff}).
Section~\ref{s:proof}
is devoted to the proof of Theorem~\ref{t:criticalSIR};
to this end, we define a time-changed version
of the epidemic and use the modified process to prove concentration of
measure estimates for various quantities of interest. In
Section~\ref{Spftakeoff}, we prove Theorem~\ref{t:takeoff}.
In Appendix \ref{app:cpts}, we state and prove a new result concerning the
size of the
giant component in the supercritical random (multi)graph with a given degree
sequence.

\section{Model, notation, assumptions and results}\label{s:notationresults}

Let $n \in \N$ and let $(d_i)_{i = 1}^n= (d_i^{(n)})_{i = 1}^n$ be a given sequence of non-negative integers.
Let $\Gr = \Gr(n, (d_i)_{i = 1}^n )$ be a simple graph (no loops or multiple edges) with $n$ vertices, chosen uniformly at random  subject to vertex $i$
having degree $d_i$ for $i=1, \ldots, n$, tacitly assuming there is any such graph at all ($\sum_{i = 1}^n d_i$ must be even, at least). For each $k \in {\mathbb Z}^+$, let $n_k$ denote
the total number of vertices with degree $k$. %by $n_k \= \#\{i:d_i = k\}$.

Given the graph $\Gr$, the epidemic evolves as a continuous-time Markov chain.  Each vertex is either
susceptible, infective or recovered.  Every infective vertex recovers at rate $\rho_n \ge 0$ and also infects each susceptible neighbour at rate $\beta_n  > 0$.

Let $\nS$, $\nI$, and $\nR$ denote the initial numbers of susceptible, infective and recovered vertices, respectively.
Further, let $\nSk$, $\nIk$ and $\nRk$ respectively, be the number of these vertices with degree $k \ge 0$.
Thus, $\nS + \nI + \nR = n$ and $\nS = \sum_{k=0}^\infty \nSk$, $\nI = \sum_{k=0}^\infty \nIk$, $\nR = \sum_{k=0}^\infty \nRk$, and $\nk = \nSk + \nIk + \nRk$.  We assume that this information is given with the degree sequence.
%  $(= \nSk + \nIk + \nRk)$.
Note that all these quantities (as well as many of the quantities introduced
below) depend on $n$. To lighten the notation, we usually do not indicate
the $n$ dependence explicitly.

\begin{remark}
We allow $\nR>0$, \ie, that  some vertices are ``recovered'' (\ie, immune)
already when we start.
It is often natural to take $\nR=0$, but
one application of $\nR>0$ is to study the effect of vaccination; this was done
in a related situation in \cite{JansonLuczakWindridgeSIR} and we leave
the corresponding corollaries of the results below to the reader.
Note that initially recovered vertices are not themselves affected by the
epidemic, but they influence the structure of the graph and thus the course
of the epidemic, so we cannot just ignore them.
\end{remark}

The basic reproductive ratio $\cR_0$ is commonly used in the context of
epidemic models, and defines the average number of new cases created by a
case of infection.
In analogy with the limiting case in
\cite[(2.23)]{JansonLuczakWindridgeSIR},
for the SIR epidemic on a random graph with a given degree sequence, we define
\begin{equation}\label{R0}
{\mathcal R}_0 = {\mathcal R}_0^{(n)}
:= \frac{\beta_n}{\rho_n + \beta_n}
\frac{\sum_{k=0}^{\infty} (k-1)k\nSk}{\sum_{k=0}^{\infty} k n_k}.
\end{equation}
Here, the probability that an infective half-edge infects another
half-edge before recovering is $\frac{\beta_n}{\rho_n + \beta_n}$, and the
average increase in the number of infective half-edges due to such an
infection event is $\frac{\sum_{k=0}^{\infty}
  (k-1)k\nSk}{\sum_{k=0}^{\infty} k n_k}$, and these are approximately
independent of one another, and approximately independent for different
half-edges.

Note that the basic reproductive ratio $\cR_0$ determines the approximate geometric growth rate of the disease during the early stages of the epidemic.
%Unsurprisingly then,  $\Rzero = 1$
%marks the threshold above which a macroscopic fraction of susceptibles can
%be become infected
The value $\cR_0 =1$ is therefore the threshold for the epidemic to take off in the population, in the sense that, if $\cR_0 > 1$, then a macroscopic fraction of the susceptibles can be infected~\cite{AnderssonSocialNetworks,
  Newman_2002,Volz_2007,BohmanPicollelli,JansonLuczakWindridgeSIR}. Here we
will consider the case where $\cR_0 = 1 + \omega (n) n^{-1/3}$,
with $\omega (n)$ tending to infinity slowly  (slower than
$n^{1/3}$) as $n \to \infty$.

It turns out that, rather than working with the quantity $\cR_0-1$, it is easier to work with  a quantity $\alpha_n$ defined by
\begin{equation}\label{e:alphan}
\alpha_n \= - (1 + \rho_n/\beta_n)\sumk k \nk / \nS + \sumk k(k-1)\nSk/\nS.
\end{equation}
Note that
\begin{equation}\label{aR0}
\alpha_n =
(\cR_0-1)
\frac{\rho_n + \beta_n}{\beta_n} \frac{\sum_{k=0}^{\infty} k n_k}{\nS}.
\end{equation}
Our assumptions below imply that
$1\le \frac{\rho_n +  \beta_n}{\beta_n} =O(1)$  %$n \to \infty$,
and that $\frac{\sum_{k=0}^{\infty} k n_k}{\nS}$ is bounded and bounded away
from 0 as $n \to \infty$, see \refR{Rrhobeta}.
Hence $(\cR_0-1)\alpha_n^{-1}$ is bounded and bounded away
from 0, and so $\alpha_n$ is equivalent to $\cR_0-1$ as a
measure of distance from criticality; see further \eqref{aR00}.
In particular, we could rephrase our assumptions and results in terms
of $\cR_0-1$ instead of $\alpha_n$, but it seems that the
mathematics works out more cleanly using $\alpha_n$.
Also, we can expect an initial growth
if and only if $\alpha_n>0$.

We consider asymptotics as $n \to \infty$, and all unspecified limits below are as $n \to \infty$. Throughout the paper we use the notation $\op$ in a standard way, as in~\cite{janson2011probabilitynotation}. That is, for a sequence of
random variables $(Y\nn)_1^\infty$ and real numbers
$(a_n)_1^\infty$, `$Y\nn = o_p(a_n)$'  means $Y\nn/a_n \pto 0$.
Similarly, $Y\nn=\Op(1)$ means that, for every $\eps>0$, there exists $K_\eps$
such that $\P(|Y\nn|>K_\eps)<\eps$ for all $n$.
% where $\pto$ signifies convergence in probability.
%
%For a sequence $(Y\nn_t)_1^\infty$ of real-valued stochastic processes defined on a subset $E$
%of $\R$ and a real-valued function $y$ on $E$,
%$Y\nn_t \pto y(t)$ uniformly on $E' \subseteq E$ if $\sup_{t \in E'}\abs{Y\nn_t - y(t)}
%\pto 0$.
%
Given a sequence of events $(A_n)_1^\infty$, $A_n$ is said to hold \whp{}
(with high probability) if $\P(A_n) \to 1$.

\medskip

Our assumptions are as follows. (See also the remarks below.)
Let $\DSn$ denote the degree of a randomly chosen susceptible vertex, so
$\Prob(\DSn = k) = \nSk/\nS$ for each $k \ge 0$.
\begin{enumerate}
\renewcommand{\theenumi}{{\upshape{(D\arabic{enumi})}}}
\renewcommand{\labelenumi}{\theenumi}

\item \label{d:first}\label{d:asympsuscdist}
%The degree of a randomly chosen initially susceptible vertex
$\DSn$
converges in distribution to a probability distribution $(p_k)_{k =
  0}^\infty$ with a finite and positive mean $\lambda:=\sumk kp_k$, i.e.
\begin{equation}\label{e:nSktopk}
\frac{\nSk}{\nS} \to p_k, \quad \; k \ge 0.
\end{equation}

\item \label{d:cubeUI}
The third power $\DSn$
%of the degree of a randomly chosen susceptible vertex
is uniformly integrable as $n\to\infty$.  That is,
given $\eps > 0$, there exists $M > 0$ such that, for all $n$,
 \begin{equation}\label{e:cubeUI}
 \sum_{k > M}\frac{k^3\nSk}{\nS} <  \eps.
 \end{equation}

\item \label{d:2ndmoment}
The second moment of the degree of a randomly
  chosen vertex is uniformly bounded, i.e.\
$\sumk k^2\nk = O(n)$.

\smallskip
 \item
\label{d:alphan}
As $n \to \infty$,
\begin{equation}
\label{e:alphanasy}
\alpha_n \to 0 \quad \mathrm{and} \quad \nS\alpha_n^3 \to \infty.
\end{equation}

\smallskip

\item \label{d:initinfective}
The total degree $\sumk k\nIk$ of initially infective vertices satisfies
\begin{equation}\label{emm}
  \sumk k\nIk  =o(n),
\end{equation}
and the limit
\begin{equation}\label{e:HI}
\nu \= \lim_{n \to \infty}\frac{1}{\nS\alpha_n^2}\sumk k\nIk
\in [0,\infty]
\end{equation}
exists  (but may be $0$ or $\infty$).
Furthermore, either $\nu=0$ or
\begin{equation}\label{RdImax}
\dImax:=\max\set{k:\nIk \ge 1}=o\Bigpar{\sumk k\nIk  }.
\end{equation}

%\smallskip

\item \label{d:p1orrhopos}
We have $p_0 + p_1 + p_2 < 1$.
%Either $p_0+p_2 <1$,
%or $\limsup_{n\to\infty}\rho_n/\beta_n > 0$,
%or $\limsup_{n\to\infty}\sumk k\nRk/\nS > 0$.

\smallskip
\item \label{d:ns}
$\liminf_{\ntoo} n_S/n>0$.
\label{d:last}
\end{enumerate}

\medskip

We will repeatedly use the fact that \ref{d:cubeUI} implies that there
exists a constant $c_0$ such that, for all $n$,
\begin{equation}\label{k3}
  \sumk k^3 \nSk = \nS \E \DSn^3 \le c_0 n.
\end{equation}

\begin{remark}\label{r:UI}
Assumption \ref{d:asympsuscdist} says $\DSn \dto \DSl$, where $\DSl$ has
distribution $(p_k)_{k = 0}^\infty$.
Given \ref{d:asympsuscdist}, assumption \ref{d:cubeUI}  is equivalent to
$\E\DSn^3 \to \E \DSl^3 < \infty$.
Furthermore, \ref{d:cubeUI} implies uniform integrability of $\DSn$ and
$\DSn^2$, so $\E\DSn \to \E \DSl$ and $\E\DSn^2 \to \E \DSl^2$.
Assumptions \ref{d:cubeUI} and \ref{d:ns} further imply that %$\dSmax = o(n^{1/3})$.
\begin{equation}\label{RdSmax}
\dSmax:=\max\set{k:\nSk \ge 1}=o(\nS^{1/3}).
\end{equation}
\end{remark}

Using the notation in \refR{r:UI},
$\gl=\E\DSl$. Furthermore,
let
\begin{align}
\glii &\= \sumk k(k-1)p_k=\E\DSl(\DSl-1),   \label{glii}
\\
\gliii& \= \sumk k(k-1)(k-2)p_k= \E\DSl(\DSl-1)(\DSl-2).   \label{gliii}
\end{align}
Then, the uniform integrability \ref{d:cubeUI} of $\DSn^3$ implies
$\glii,\gliii<\infty$ and furthermore
\begin{align}
\label{glii2}
\glii &= \lim_{n\to\infty}\E\DSn(\DSn-1)
=\lim_\ntoo\sumk k(k-1)\frac{\nSk}{\nS},
\\\label{gliii2}
\gliii &= \lim_{n\to\infty}\E\DSn(\DSn-1)(\DSn-2)
=\lim_\ntoo\sumk k(k-1)(k-2)\frac{\nSk}{\nS}
\end{align}
Also, $\gl,\glii,\gliii > 0$ by~\ref{d:p1orrhopos}.

\medskip

%It will be convenient for us to work with multigraphs, that is to allow
%loops and multiple edges. More precisely,
Let $\GrCM=G^* (n, (d_i)_1^n )$ be the
random multigraph with given degree sequence $(d_i)_1^n$ defined by the
{\it configuration model}: we take a set of $d_i$ half-edges for each vertex $i$
and combine half-edges into edges by a uniformly random matching (see \eg{}
\cite{bollobas}). Conditioned on the multigraph being simple, we obtain  $G
= G (n, (d_i)_1^n )$, the uniformly distributed random graph with degree
sequence $(d_i)_1^n$.
The configuration model %as a device for studying the uniform simple graph (as we do)
has been used in the study of epidemics in a number of earlier works, see,
for example,
\cite{andersson1998limit,Ball200869,BrittonJansonMartinLof,DDMT12,BohmanPicollelli}.
As in many other papers, including \cite{JansonLuczakWindridgeSIR},
we prove our results for the SIR epidemic on $\GrCM$, and, by conditioning
on $\GrCM$ being simple, we then deduce that these results also hold for the SIR
epidemic on $\Gr$. The results below thus hold for both the random multigraph
$\GrCM$ and the random simple graph $\Gr$.

This argument relies on the probability that
$\GrCM$ is simple
being bounded away from zero as  $n \to \infty$;
by the main theorem of \cite{Janson:2009:PRM:1520305.1520316}
(see also \cite{simpleII})
this occurs
provided condition \ref{d:2ndmoment} holds.
Most of the results below are of the ``\whp'' type
(or can be expressed in this form); then this transfer to the simple graph
case is routine and will not be commented on further. The exception is
\refT{t:takeoff}\ref{takeoff3}, where we obtain a limiting probability
strictly between 0 and 1, and we therefore need a more complicated argument,
see \refS{Spftakeoff}; we also use an extra assumption in this case.

We now state our main result, that, under the conditions above,
the epidemic is either very small, or of a size at least approximatively
proportional to $n\ga_n$ (and thus to $n(\cR_0-1)$).
%, with an explicit proportionality constant.
As just said, the theorem holds for both the multigraph $\GrCM$
and the simple graph $\Gr$.

\begin{theorem}\label{t:criticalSIR}
Suppose that \ref{d:first}--\ref{d:last} hold.

Let $\fs$ be the total number of  susceptible vertices that ever get
infected.
\begin{romenumerate}
\item \label{tsizei}
If\/ $\nu=0$, then there exists a sequence $\eps_n\to0$ such that,
for each $n$, \whp{} one of the following holds.
\begin{enumerate}%[\upshape(a)]
 \renewcommand{\labelenumii}{\textup{(\alph{enumii})}}%
 \renewcommand{\theenumii}{\textup{(\alph{enumii})}}%
\makeatletter\renewcommand{\p@enumii}{}\makeatother
\item \label{tsizei:small}
$\fs/\nS \alpha_n < \eps_n $ (the epidemic is small and ends prematurely).
\item \label{tsizei:large}
$\abs{\fs/\nS \alpha_n  -2\gl/\gliii} < \eps_n$
(the epidemic is large and its size is well concentrated).
\end{enumerate}

\item \label{tsizeii}
If\/ $0<\nu<\infty$, then
$\fs/\nS \alpha_n  \pto \gl\xxl$.

\item \label{tsizeiii}
If\/ $\nu=\infty$, then
\begin{equation}
  \frac{\fs}{\bigpar{\nS\sumk k\nIk}\qq} \pto
\frac{\sqrt2\,\gl}{\sqrt{\gliii}}
\end{equation}
\end{romenumerate}

Moreover, in cases \ref{tsizei}\ref{tsizei:large}, \ref{tsizeii} and
\ref{tsizeiii}, the following holds.
Let $\fs_k$ be the number of degree $k\ge 0$ susceptible vertices that
ever get infected. Then
\begin{align}
\sumk \biggabs{\frac{\fs_k}{\fs} &- \frac{kp_k}{\gl}  }
\pto 0.
\label{t2ab}
\end{align}
\end{theorem}

Thus, \eqref{t2ab} says that, except in the case
\ref{tsizei}\ref{tsizei:small},
the total variation distance between the degree distribution $(\fs_k/\fs)$
of the vertices that get infected and the size-biased
distribution $(kp_k/\gl)$ converges to 0 in probability.

Note that case \ref{tsizei} of \refT{t:criticalSIR} says that, for a range
of initial values of the number of infective half-edges (viz.\ when $\nu=0$),
if the epidemic takes off at all, then it has approximately the size
$(2\gl/\gliii)\nS\ga_n$.
Hence, in this range, the size of the epidemic does (to the first order) not
depend on the initial number of infective half-edges (only the probability
of a large outbreak does), so this can be seen as
the ``natural'' size of an epidemic.
This also means that in this range, most of the outbreak can be traced back
to a single initial infective half-edge.

However, when the initial number of infective half-edges number gets larger, the many small outbreaks coming from the different initially infective half-edges will add up to
a substantial outbreak. So there is a threshold where this bulk of combined
small outbreaks is of about the same size as the ``natural'' size of a large
outbreak. The value $\nu$ is, in the limit as $n \to \infty$, the ratio of the initial number divided by this threshold, so it shows, roughly, whether the combined small outbreaks
give a large contribution to the outbreak or not.
Our theorem then shows that, if the initial number of infective
half-edges is larger (to be precise, $\nu>0$), then they force a larger
outbreak, with a size that is proportional to the square root of the initial
number of infective half-edges in the range $\nu=\infty$.
(For $0<\nu<\infty$, there is a smooth transition between the two extremal
cases.)

The following result gives conditions for the occurrence of a large epidemic
in \refT{t:criticalSIR}\ref{tsizei}. In anticipation of later notation,
let $\XI0:= \sumk k \nIk$ be the total degree of
initially infective vertices (i.e.\ the total number of initially infective
half-edges).

\begin{theorem}\label{t:takeoff}
Suppose that the assumptions of \refT{t:criticalSIR} are
satisfied with $\nu=0$.
% and let $\XI0:= \sumk k \nIk$ be the total degree of
%initially infective vertices (i.e.\ the total number of initially infective
%half-edges).
\begin{romenumerate}
\item \label{takeoff1}
If $\alpha_n  \XI0 \to 0$, then
$\fs = \op(\alpha_n\qww)= \op(\nS \alpha_n)$, and thus
case \ref{tsizei}\ref{tsizei:small}
in \refT{t:criticalSIR}
occurs \whp.

\item \label{takeoff2}
If $\alpha_n \XI0 \to \infty$ then case \ref{tsizei}\ref{tsizei:large}
in \refT{t:criticalSIR}
occurs \whp.

\item \label{takeoff3}
Suppose that $\alpha_n \XI0$ is bounded above and below.
In the simple graph case, assume also that
$\sum_{k\ge1} k^2 \nIk=o(n)$
and
$\sum_{k\ge\gan\qw} k^2 \nRk=o(n)$.
Then both
cases \ref{tsizei}\ref{tsizei:small} and \ref{tsizei}\ref{tsizei:large}
in \refT{t:criticalSIR}
occur with probabilities bounded away from $0$ and $1$.
Furthermore, if $\dImax=o(\XI0)$, then the probability that
case \ref{tsizei}\ref{tsizei:small} in \refT{t:criticalSIR} occurs is
\begin{equation}\label{takeoff}
 \exp\Bigpar{-\frac{\glii+\gl +  \sumk k \nRk/\nS}{\glii\gliii} \ga_n \XI0}
+o(1).
\end{equation}
Moreover, in the case the epidemic is small, $\fs=\Op\bigpar{\gan\qww}$.
\end{romenumerate}
\end{theorem}
Note that $\sumk k \nRk/\nS$ in \eqref{takeoff} is bounded because of
\ref{d:2ndmoment} and \ref{d:ns},
and that
\eqref{takeoff} holds in cases \ref{takeoff1} and
\ref{takeoff2} too.
A more complicated formula extending \eqref{takeoff} holds also in the case
when the condition
$\dImax=o(\XI0)$ fails, see \eqref{takeoffDeluxe} in \refR{Rtakeoff}.

\begin{remark}
The quantity $\nu\ge0$ controls the initial number of infective contacts.
If $\nu > 0$,
so a large epidemic occurs by \refT{t:criticalSIR},
then
\[
\gan\XI0=
\alpha_n \sumk k \nIk = (\nS \alpha_n^3) \frac{\sumk k \nIk}{\nS \alpha_n^2} \to \infty,
\]
by \eqref{e:alphanasy} and \eqref{e:HI};
hence the condition in \refT{t:takeoff}\ref{takeoff2} holds automatically
when $\nu>0$.
\end{remark}

\begin{remark}
  The condition \eqref{emm} that the total degree of initally infective vertices
  is $o(n)$ is, by \ref{d:2ndmoment} and the Cauchy--Schwarz inequality,
  equivalent to $n_I=o(n)$, at least if we ignore isolated infective vertices.
Note that the opposite case, when $n_I/n$ has a strictly positive limit, is
treated in \cite[Theorems 2.6 and 2.7]{JansonLuczakWindridgeSIR} (under otherwise similar
assumptions).
\end{remark}

\begin{remark}
  The assumption \eqref{RdImax} (which is required only when $\nu>0$) says
  that no single infective vertex has a significant fraction of the total
  infective degree.
%(Otherwise, the course of the epidemic might be too much
%  influenced by the
\end{remark}

\begin{remark}\label{Rrhobeta}
Assuming \ref{d:asympsuscdist} and \ref{d:cubeUI},
the assumption $\gan\to0$ in \ref{d:alphan} is equivalent to $\cR_0\to1$, as
said above.
To see this, note that
\ref{d:asympsuscdist} and \ref{d:cubeUI} imply
(see \refR{r:UI})
\begin{equation}\label{lha}
  \sumk k n_k/\nS \ge
  \sumk k\nSk/\nS
= \E \DSn  \to \E\DSl=\gl>0.
\end{equation}
If $\gan\to0$, then \eqref{lha} and \eqref{aR0} imply that $\cR_0\to1$.

Conversely, still assuming \ref{d:asympsuscdist} and \ref{d:cubeUI},
if $\cR_0\to1$, then it follows easily from \eqref{R0} that
\begin{equation}
  \label{rhobeta}
\rho_n/\beta_n =O(1),
\end{equation}
and also that
\begin{equation}
  \label{lhb}
\sumk kn_k = O\Bigpar{\sumk (k-1)k\nSk} = O(\nS).
\end{equation}
Hence, \eqref{aR0} implies that $\gan\to0$.

To be precise, \eqref{aR0} and \eqref{R0} yield by \eqref{glii2} and
$\cR_0\to1$,
\begin{equation}\label{aR00}
\alpha_n =
\frac{\cR_0-1}{\cR_0}
\frac{\sum_{k=0}^{\infty}(k-1) k \nSk}{\nS}
=\bigpar{1+o(1)}\glii(\cR_0-1).
\end{equation}

Note that by combining the two parts of the argument,
we have shown that our assumptions
\ref{d:asympsuscdist}, \ref{d:cubeUI} and \ref{d:alphan} imply
\eqref{rhobeta}
and the complementary bounds \eqref{lha} and \eqref{lhb}.
(This can also easily be seen using \eqref{e:alphan}.)
\end{remark}

\begin{remark}
%it follows from  \ref{d:cubeUI} that
%$\sum_{k=0}^{\infty} (k-1)k\nSk/\nS=O(1)$.
%Hence, \eqref{e:alphan} and \ref{d:alphan} imply
%\begin{equation}\label{lh}
%(1 + \rho_n/\beta_n)\sumk k \nk / \nS =-\alpha_n +O(1) =O(1)
%\end{equation}
%and thus $\sumk kn_k/\nS=O(1)$.
We saw in \refR{Rrhobeta} that
\ref{d:asympsuscdist}, \ref{d:cubeUI} and \ref{d:alphan} imply
\eqref{lhb}.
Since $n-n_0\le\sumk kn_k$, it follows that
$n-n_0=O(\nS)$.
Hence,
assumption \ref{d:ns} is needed only to the exclude the rather trivial
case that almost all
of the population consist of isolated infective vertices, which cannot
spread the epidemic.
Note also that \ref{d:ns} implies that it does not matter whether we use
$\nS$ or $n$ in estimates such as
\eqref{RdSmax}.
\end{remark}

\subsection{$G(n,p)$ and $G(n,m)$}

The results above apply to the graphs $\gnp$ and $\gnm$
by conditioning on the sequence of vertex degrees (which are now random),
since given the  vertex degrees, both
$\gnp$ and $\gnm$ are uniformly distributed over all (simple) graphs with
these vertex degrees. Moreover, if \ntoo{} and $p\sim\gl/n$, or $m\sim n\gl/2$,
for some $\gl>0$, then the degree distribution is asymptotically Poisson
$\Po(\gl)$. For \gnp, this leads to the following result.

\begin{corollary}
Suppose that $\beta_n > 0$ and $\rho_n \ge 0$ for each $n \in {\mathbb
  N}$. Let $\lambda \ge 1$, and assume that $\frac{\beta_n +
  \rho_n}{\beta_n} \to \lambda$ as $n \to \infty$. Let $\eta_n \to 0$, and
consider the SIR epidemic on the random graph $G(n, \frac{\lambda (1+
  \eta_n)}{n})$ with infection rate $\beta_n$ and recovery rate $\rho_n$.
Suppose that there are $\nI= o(n)$ initially infective vertices chosen at
random, and all the other vertices are susceptible.
Let
\begin{equation}
\gamma_n
:= 1 -\frac{\beta_n + \rho_n}{\lambda \beta_n}+\eta_n
- (1+\eta_n)\frac{\nI}{n}.
\end{equation}
Then $\gamma_n \to 0$. Assume that $n \gamma_n^3 \to \infty$,
and that $\mu = \lim \frac{\nI}{n \gamma^2_n}$ exists.
\begin{romenumerate}
\item \label{ERtsizei-1}
If\/ $\mu=0$, then there exists a sequence $\eps_n\to0$ such that
for each $n$, \whp{} one of the following holds.
\begin{enumerate}%[\upshape(a)]
 \renewcommand{\labelenumii}{\textup{(\alph{enumii})}}%
 \renewcommand{\theenumii}{\textup{(\alph{enumii})}}%
\makeatletter\renewcommand{\p@enumii}{}\makeatother
\item \label{ERtsizei:small-1}
$\fs/(n \gamma_n) < \eps_n $.
\item \label{ERtsizei:large-1}
$\abs{\fs/(n \gamma_n) -2} < \eps_n$.
\end{enumerate}
Moreover, the probability that \ref{ERtsizei:small-1} holds
is
\begin{equation}\label{takeoffgnp}
  \exp\bigpar{-(1+\gl\qw)\gam_n n_I} +o(1).
\end{equation}
In particular, \ref{ERtsizei:small-1} holds \whp{} if\/ $\gam_n\nI\to0$ and
\ref{ERtsizei:large-1} holds \whp{} if\/ $\gam_n\nI\to\infty$.
\item \label{tsizeii-1}
If\/ $0<\mu <\infty$, then
$\fs/n \gamma_n  \pto 1 + \sqrt{1 + 2\mu }$.

\item If $\mu = \infty$, then
$$\frac{\fs}{(\nS \nI)^{1/2}} \pto \sqrt{2}.$$
\end{romenumerate}
The same holds for $\gnm$ with $m=n\gl(1+\eta_n)/2$.
\end{corollary}

\begin{proof}
As said above, we condition on the vertex degrees.
We have $\nSk/\nS \pto p_k:=\P(\Po(\gl)=k)$ for every $k$;
for convenience, we use the \Skorohod{} coupling theorem
\cite[Theorem 4.30]{kallenberg}
so we may assume that
this holds \as{} for each $k$; thus \eqref{e:nSktopk} holds a.s.
Similarly we may assume that $\sum_k k^4\nk/n$ converges \as{},
%(to $\gl^4+6\gl^3+7\gl^2+\gl$),
and then
\ref{d:cubeUI} and \ref{d:2ndmoment} hold a.s.
Furthermore, $\gan$ is now random, and it is easy to see from
\eqref{e:alphan} that
\begin{equation}
  \begin{split}
\frac{\nS}{n}\gan&
=-\frac{\gb_n+\rho_n}{\gb_n}(1+\eta_n)\gl
+\bigpar{(1+\eta_n)\gl}^2\Bigpar{1-\frac{\nI}n}+O_p\bigpar{n\qqw}
\\&
=(1+\eta_n)\gl^2 \gam_n+O_p\bigpar{n\qqw}
=\bigpar{\gl^2+\op(1)}\gam_n.
  \end{split}
\end{equation}
Repeating the \Skorohod{} trick, we may thus assume also that
$\ga_n/\gam_n\to\gl^2$. Similarly we may assume
$\XI0=\sum_k k\nIk=\bigpar{1+O(\nI\qqw)}\gl{\nI}$, and then \eqref{e:HI} holds
with
$\nu=\mu/\lambda^3$; it is also easy to see that \eqref{RdImax} may be
assumed. Then all the conditions \ref{d:first}--\ref{d:last} hold \as,
and the result follows as a consequence of Theorems \ref{t:criticalSIR} and
\ref{t:takeoff},
noting that $\DSl\sim\Po(\gl)$, and thus $\glii=\gl^2$ and $\gliii=\gl^3$.
\end{proof}

\section{Proof of \refT{t:criticalSIR} }\label{s:proof}

\subsection{Simplifying assumptions}\label{s:simplify}

We assume for convenience that $\nI = o(n)$.
In fact, we may assume that $\nIx0=0$ by deleting all initially infective
vertices  of degree 0, since these are irrelevant; then $\nI=o(n)$ as a
consequence of \eqref{emm}.
Note that this will not affect
$\cR_0$, $\ga_n$, $\nu$ or the other constants and assumptions above.

Similarly, we assume
that initially there are no recovered vertices, that is $\nR =
0$. It is easy to modify the proofs below to handle the case
$\nR \ge 1$.
Alternatively,
we may observe that our results in the case $\nR = 0$ imply the
corresponding results for general $\nR$
by the following argument.
(See \cite{Janson:2009:SJ215} for similar arguments in a related situation.)
We replace
each initially
recovered vertex of degree $k$  by $k$ separate
susceptible vertices of degree 1, so there are a total of $\XR0:=\sumk k\nRk$
additional ``fake'' susceptible vertices of degree 1;
this will not change the course of the epidemic
(in the multigraph case)
except that some of these
fake susceptible vertices will be infected. (Note that they never can infect
anyone else.)
The alteration will not affect $\cR_0$,
although  $\ga_n$ and the asymptotic distribution $(p_k)$ will be
modified. Note that $\XR0=O(\nS)$ by \ref{d:2ndmoment} and \ref{d:ns};
by considering suitable subsequences we may thus assume that
$\XR0/\nS\to r$ for some $r\in[0,\infty)$.
It is easy see that the modified degree distribution satisfies all the
assumptions above and that,
if we use a prime to indicate quantities after the replacement,
then
$\nS'=\nS+\XR0\sim (1+r)\nS$,
$\gan'\sim \gan/(1+r)$,
$\nS'\gan'=\nS\gan$,
$\nu'=(1+r)\nu$,
$p_1'=(p_1+r)/(1+r)$,
$p_k'=p_k/(1+r)$ for $k\neq1$,
$\gl'=(\gl+r)/(1+r)$,
$\glii'=\glii/(1+r)$,
$\gliii'=\gliii/(1+r)$.

If case \ref{tsizei}\ref{tsizei:small} in \refT{t:criticalSIR}
occurs for the modified process, it occurs for the original process too,
since $\fs\le\fs'$, and there is nothing more to prove.

In the other cases,
we have $\fs'\to\infty$ \whp.
We note that of the $\nSx1'=\nSx1+\XR0$ susceptible
vertices of degree 1, $\XR0$ are fake.
Conditioned on the number $\fs_1'$ of susceptible vertices of degree 1 that
get infected, the number $\fs'-\fs=\fs_1'-\fs_1$ of fake susceptible
vertices that get infected has a hypergeometric distribution,
and, using \eg{} Chebyshev's inequality,
it follows that \whp{}
(leaving the simple modification when $p_1=r=0$ to the reader)
\begin{equation}
  \fs'-\fs=\fs_1'-\fs_1=\frac{\XR0}{\nSx1+\XR0}\fs_1'+o(\fs')
=\frac{r}{p_1+r}\fs_1'+o(\fs').
\end{equation}
By \eqref{t2ab} and the relations above, this yields \whp{}
\begin{equation}
  \fs'-\fs
=\frac{r}{p_1+r}\fs_1'+o(\fs')
=\frac{r}{p_1+r}\frac{p_1'}{\gl'}\fs'+o(\fs')
=\frac{r}{\gl+r}\fs'+o(\fs').
\end{equation}
Consequently, \whp{} $\fs/\fs'=\xqfrac{\gl}{\gl+r}+o(1)$.

It is then easy to check that \refT{t:criticalSIR} and
Theorem~\ref{t:takeoff} for the original process both follow from these
results in the case with no initially recovered vertices.

We make these  simplifying assumptions $\nI=o(n)$ and $\nR=0$ throughout
this section (and the following one), in addition to
\ref{d:first}--\ref{d:last}. In particular, $\nI+\nR=o(n)$, and thus \ref{d:ns} is
strengthened to
\begin{equation}\label{nns}
\nS/n \to 1.
\end{equation}
%so that we may and will always write $n$ in place of $\nS$.
% without changing \ref{d:first}--\ref{d:last} or our result. % we can't emphasise this enough!

We may also assume $\ga_n>0$, by ignoring some small $n$ if necessary.
Finally, recall that in the proofs we first consider the random multigraph $\GrCM$.

\subsection{Time-changed epidemic on a random multigraph}
\label{s:acceleratedepidemic}

We first study the epidemic on the configuration model multigraph $\GrCM$, revealing its edges (i.e.\ pairing off the half-edges) while the epidemic spreads, as in~\cite{JansonLuczakWindridgeSIR} (see other variants in~\cite{andersson1998limit,Ball200869,DDMT12,BohmanPicollelli}).
%More precisely, we will study the following Markovian joint evolution
%of the multigraph and epidemic
%(the dynamics are as in \cite{JansonLuczakWindridgeSIR},  and variants are used in \cite
%{andersson1998limit,Ball200869,DDMT12,BohmanPicollelli}).
We call a half-edge susceptible,  infective or recovered according to the type of vertex it is attached to.
Unpaired half-edges are said to be {\em free}. Initially, each vertex $i$ has
$d_i$ half-edges and all of them are free.

Each free infective half-edge chooses a
\fhe{} at rate $\beta_n > 0$, uniformly at random from among all the other free half-edges.  Together the pair form an edge, and are removed from the set of free half-edges.
If the chosen \fhe{} belongs to a susceptible vertex then that vertex
becomes infective.
Infective vertices
recover at rate $\rho_n \ge 0$.

We stop the process when no infective \fhe{}s remain, which is the time
when the epidemic stops spreading.  Some infective vertices may remain but
they trivially recover at \iid{}  exponential times.
Some free susceptible and recovered half-edges may also remain.  These could
be paired uniformly to reveal the remaining edges in $\GrCM$, if desired.
However, this step is irrelevant for the course of the epidemic.

In order to prove our results, we perform a time change in the process:
when in a state with $\xI \ge 1$ free infective half-edges, and a total of $x$ \fhe{}s of any type, we multiply all transition rates by $(x-1)/\beta_n \xI$  (this multiple is at least $1/(2\beta_n)$, since $\xI \ge 1$ implies that $x \ge 2$).
Then each free susceptible half-edge gets infected at rate 1,  each infective
vertex recovers at rate $\rho_n(x-1)/\beta_n\xI$, and each free infective
half-edge  pairs off at rate $(x-1)/\xI$.

In the time changed process, let $\Sv{t}$, $\Iv{t}$ and $\Rv{t}$ denote the
numbers of susceptible,
infective and recovered vertices, respectively, at time $t \ge 0$.
Let $\Svk{t}$ be the number of susceptible vertices of degree $k\ge 0$ at time $t$.
Then $\Sv{t} = \sumk \Svk{t}$ is decreasing and $\Rv{t}$ is increasing in $t$.
Moreover, $\Svk{0} = \nSk$, $\Iv{0} = \nI$ and $\Rv{0} = \nR = 0$. % \nR$.

Also, we let $\XS{t}$, $\XI{t}$ and $\XR{t}$ be the numbers of free susceptible, infective and recovered half-edges, respectively, at time $t$.  Then $\XS{t} = \sumk k \Svk{t}$ is decreasing,
$\XS{0} = \sumk k \nSk$,  $\XI{0} = \sumk k \nIk$ and $\XR{0} = 0$ (by our
simplifying assumptions in \refS{s:simplify}).

We denote the duration of the time-changed epidemic by
\begin{equation}\label{e:iEnd}
 \iEnd \= \inf\{t \ge 0  : \XI{t} = 0 \}.
\end{equation}
At time $\iEnd$, we simply stop, as said above.
(The last infection may have occurred somewhat earlier, since the last
free infective half-edge may have recovered or paired of with an infective
or recovered half-edge. It follows \eg{} from \eqref{e:XSLLN} below that the
last actual infection \whp{} did not happen much earlier, but this is
irrelevant for our results, and we use \eqref{e:iEnd} as the definition.)

\subsection{Concentration of measure}
We will show that $\Sv{t}(k)$, $\XS{t}$, $\XI{t}$ and $\XR{t}$ are uniformly close to
 certain deterministic functions.  Let

% \begin{equation}
% \label{e:fSvnt}
%  \fSvn(t) \= \sumk \nSk e^{-kt},
% \end{equation}
\begin{align}
\label{e:fXSnt}
 \fXSn(t) &\= \sumk k \nSk e^{-kt},
\\
\label{e:fXRnt}
 \fXRn(t) &\= %e^{-t} \sumk k\nRk +
  \frac{\rho_n}{\beta_n}e^{-t}(1 - e^{-t})\sumk k\nk,
\\
\label{e:fXInt}
 \fXIn(t) &\= e^{-2t}\sumk k \nk - \fXSn(t) - \fXRn(t).
\end{align}

\begin{theorem}\label{t:conc}
Let $\tgan$ be any numbers with $\ga_n\le\tgan=o(1)$ such that
  \begin{equation}\label{k2I}
	\sumk k^2 \nIk = o\bigpar{n^2\tgan^4}.
  \end{equation}
Then,
for any fixed $t_0 <\infty$,
\begin{align}
\label{e:SLLN}
%  \sup_{t\le \tgan t_0 \wedge \iEnd} \abs{\Sv{t} - \fSvn(t)}
% &= \op(n\tgan^2),
 \sumk \sup_{t\le \tgan t_0 \wedge \iEnd} \abs{\Sv{t}(k) - \nSk e^{-kt}}
&= \op(n\tgan^2),
\\
\label{e:XSLLN}
 \sup_{t\le \tgan t_0 \wedge \iEnd} \abs{\XS{t} - \fXSn(t)}
&= \op(n\tgan^2),
\\
\label{e:XRLLN}
 \sup_{t\le\tgan t_0 \wedge \iEnd} \abs{\XR{t} - \fXRn(t)}
&= \op(n\tgan^2),
\\
\label{e:XILLN}
 \sup_{t\le\tgan t_0 \wedge \iEnd} \abs{\XI{t} - \fXIn(t)}
&= \op(n\tgan^2).
\end{align}
\end{theorem}

% \marginal{$\op(n\tgan)$ is enough for \eqref{e:SLLN}}

The above result establishes concentration on time intervals of length
$O(\tgan)$.
In \refS{s:duration}, we use it to show that, for a suitable choice of
$\tgan$,  the duration
of the epidemic satisfies $\iEnd=O(\tgan)$ \whp.
It follows that  the theorem then holds also with $t_0=\infty$, see \refR{Rt1}.

The remainder of this subsection contains the proof of Theorem~\ref{t:conc}.  We first need two lemmas concerning
the evolution of the number of susceptible vertices and the total number of free half-edges.

In the time-changed epidemic, each free susceptible half-edge gets infected at rate~1, until $\iEnd$.
We further modify the process so that free susceptible half-edges continue to be infected at rate 1 even when there are no more free infective half-edges.
% this could be achieved with a phantom infective for example.
%The benefit is that we can ignore the fact that our process stops at time $\iEnd$.  To be more precise,
Let $\Svkm{t}$ be the
number of susceptible individuals of degree $k$ in the modified process.
% and $\tilde\iEnd$ the duration.
Then
$(\Svkm{t\wedge \iEnd}: k \in {\mathbb Z}^+, t \ge 0)$ has the same distribution as $(\Svk{t\wedge\iEnd}: k \in {\mathbb Z}^+, t \ge 0)$,
and so, to prove \eqref{e:SLLN} and \eqref{e:XSLLN}, it suffices to prove that
\begin{equation}\label{e:SLLNx}
\sumk \sup_{t\le \tgan t_0} \abs{\Svkm{t} - \nSk e^{-kt}}
= \op(n\tgan^2),
\end{equation}
and
\begin{equation}
\sup_{t\le \tgan t_0} \abs{\tilde{X}_{S,t} - \fXSn(t)}
= \op(n\tgan^2),
\end{equation}
where $\tilde{X}_{S,t}= \sumk k\Svkm{t}$.
For each $t$, let
\begin{equation}\label{e:wtFIt}
W_t \= \sumk k^2( \Svkm{t} - \nSk e^{-kt} ).
\end{equation}

%\marginal{for the proof later we only need $\op(n\tgan)$ in \eqref{e:SLLN})}

\begin{lemma}\label{l:susdist2ndmomentBD}
Fix $t_0 <\infty$ and assume
$\ga_n\le\tgan=o(1)$.
Then $\E\sup_{t\le \tgan t_0} |W_t|=o(n \tgan)$, and hence
%\begin{equation}
%\label{e:susdist2ndmomentBD}
$$
\E\sup_{t\le \tgan t_0\wedge\iEnd}
\biggabs{\sumk k^2( \Svk{t} - \nSk e^{-kt} )} = o(n\tgan).$$
%\end{equation}
\end{lemma}

\begin{proof}[Proof of \refL{l:susdist2ndmomentBD}]

We enumerate the initially susceptible vertices as $i = 1,2,\ldots,\nS$ and
denote by $\dSi$  the degree of initially susceptible vertex $i$.   Let
$L_i$ be the time at which initially susceptible vertex $i$ becomes
infective (in the modified process).
Then each $L_i$ has exponential distribution with rate $\dSi$,
and the $L_i$ ($i = 1,2,\ldots,\nS$) are all independent of one another. It
follows that, for each fixed $t$, the random variables $\FIt \=
\dSi^2(\Indi{L_i > t} - e^{-t\dSi})$ each have mean zero and are all
independent. Note that $W_t = \sum_{i = 1}^{\nS} \FIt$.

Each $|\FIt|$ is bounded by $\dSmax^2$, where, as in~(\ref{RdSmax}),
$\dSmax = \max_i\dSi$.  Hence,
by Bernstein's inequality
for sums of bounded independent centred  random variables,
see \eg{} \cite[Theorem 2.7]{mcdiarmid98} or \cite[(2.10)]{BoucheronLM},
for each $a\ge0$,
\begin{equation}\label{e:Bernstein}
\Prob(\abs{W_t} > a) = \Prob\left(\Bigabs{\sum_{i = 1}^{\nS} \FIt} > a\right)
\le 2\exp\left( - \frac{a^2}{2\sum_{i = 1}^{\nS} \E \FIt^2 + 2a \dSmax^2/3} \right).
\end{equation}

Now, for any $t \le \tgan t_0$, using \eqref{k3},
\begin{equation}\label{e:Ber2}
  \begin{split}
2\sum_{i = 1}^{\nS} \E \FIt^2
 &= 2\sum_{i = 1}^{\nS} \dSi^4\Var\bigpar{\Indi{L_i > t}}  \le 2\sum_{i = 1}^{\nS} \dSi^4 \bigpar{1 - e^{-t\dSi}}\\
 & \le 2 \dSmax^2 t \sum_{i = 1}^{\nS} \dSi^3 % \label{e:expdenominator}
  \le 2t_0 \tgan \dSmax^2 \sum_k k^3\nSk \le 2c_0t_0 \tgan n \dSmax^2.
  \end{split}
\end{equation}
%where $c_0>0$ is a constant such that $\sum_k k^3\nSk \le c_0 n$, see
%\ref{d:cubeUI}.

Furthermore, $\tgan n^{1/3}\ge\ga_n n\qqq \to \infty$ and
so by \eqref{RdSmax},
\begin{equation}\label{dsx}
 \dSmax= o\bigpar{n^{1/3}}
=o\bigpar{(n\tgan)\qq}.
\end{equation}
Thus,  for $n$ sufficiently large,
$\dSmax\le (n\tgan)\qq$, and
then for any $u \ge 2c_0
t_0$ and  $a = u (n \tgan)^{1/2}\dSmax$, by \eqref{e:Ber2},
\begin{equation*}
  \begin{split}
\exp\left( - \frac{a^2}{2\sum_{i = 1}^{\nS} \E \FIt^2 + 2a \dSmax^2/3} \right)
&\le \exp\left( - \frac{u^2}{2c_0t_0 + 2u \dSmax/3(n\tgan)^{1/2}} \right)\\
&\le \exp\left( - \frac{u^2}{2c_0t_0 + u} \right)\le \exp\left( - u/2\right).
  \end{split}
\end{equation*}
%In other words,
Hence, by~\eqref{e:Bernstein}, for
$n$ sufficiently large and for each
each $t \le t_0 \tgan$ and $u \ge 2c_0t_0$,
\begin{equation}
  \label{e:B3}
\Prob(\abs{W_t} > u (n \tgan)^{1/2}\dSmax) \le 2\exp\left( - u/2\right).
\end{equation}
Note also that $(n \tgan)^{1/2} \dSmax = o(n \tgan)$ by \eqref{dsx}. Let
$\omega_n$ be an integer valued function such that $\omega_n \to \infty$ and
$(n \tgan)^{1/2} \dSmax\omega_n  = o(n \tgan)$. We divide the interval
$[0,t_0 \tgan]$ into $\omega_n$ subintervals $[\tau_l,\tau_{l+1}]$, where
$\tau_l = l t_0\tgan/\omega_n$ for $l=0,\ldots, \omega_n-1$.

Since $\Svkm{t}$ and $e^{-kt}$ are both decreasing in $t$, each of the sums
$\sumk k^2 \Svkm{t}$ and $\sumk k^2 \nSk e^{-kt}$ is also decreasing in $t$.  Thus, for any $0 \le l <  \omega_n$,
\begin{align*}
\sup_{\tau_l \le t \le \tau_{l+1}}\abs{W_t}
& \le \lrabs{\sumk k^2(\Svkm{\tau_l} - \nSk e^{-k\tau_{l+1}})} + \lrabs{\sumk k^2(\Svkm{\tau_{l+1}} - \nSk e^{-k\tau_{l}})} \nonumber \\
& \le \abs{W_{\tau_{l}}} + \abs{W_{\tau_{l +1}}} + 2\sumk k^2 \nSk(e^{-k\tau_l} - e^{-k\tau_{l+1}})\\
& \le \abs{W_{\tau_{l}}} + \abs{W_{\tau_{l +1}}} + 2\sumk  k^3 \nSk (\tau_{l+1} - \tau_l ),% \tgan t_0 /m_n.
\end{align*}
% using \eqref{e:maxdifftl}.
and so, since $\sumk  k^3 \nSk \le c_0 n$ and
$\tau_{l+1}-\tau_l=t_0\tgan/\go_n=o(\tgan)$,
noting  $W_0=0$,
\begin{equation}\label{e:supwt}
\sup_{t \le \tgan t_0}\abs{W_t}  = \max_{l < \omega_n} \sup_{\tau_l \le t \le \tau_{l+1}}\abs{W_t}
\le 2\max_{1 \le l \le \omega_n}\abs{W_{\tau_{l}}} + o(n\tgan).
\end{equation}
Now, for $n$ sufficiently large and $u \ge 2c_0t_0$,
by \eqref{e:B3},
\begin{equation}\label{eB4}
\Prob\Bigpar{\max_{1 \le l\le\omega_n}\abs{W_{\tau_{l}}} > u(n \tgan)^{1/2}\dSmax}
\le 2\omega_n \exp(-u/2).
\end{equation}
%Since $\go_n\to\infty$,
For sufficiently large $n$,
$2\go_n\ge e^{c_0t_0}$, and then
\eqref{eB4} holds trivially for $u<2c_0t_0$ too.
Hence,
for large $n$,
\begin{align}
\E \max_{1 \le l \le \omega_n } \abs{ W_{\tau_{l}} }
& =  (n\tgan)^{1/2}\dSmax \int_{0}^\infty
\Prob\Bigpar{ \max_{1 \le l  \le \omega_n} \abs{W_{\tau_{l}}} > u (n \tgan)^{1/2}\dSmax}\,du \nonumber \\
& \le (n \tgan)^{1/2}\dSmax \int_{0}^\infty 2\go_n e^{-u/2} du  \nonumber \\
&
=4 (n \tgan)^{1/2}\dSmax \go_n
= o(n\tgan) \label{e:Emaxdev},
\end{align}
and hence also
$\E \sup_{t \le \tgan t_0}\abs{W_t} = o(n \tgan)$
by \eqref{e:supwt}.
\end{proof}

We now prove a concentration of measure result for
the total number $\X{t}$ %\= \XS{t} + \XI{t} + \XR{t}$
of \fhe{}s.

\begin{lemma}\label{l:XLLN}
For every fixed $t_0 > 0$,
and $\ga_n\le\tgan=o(1)$,
\begin{equation}
\label{e:XLLN}
 \sup_{t \le  \tgan t_0 \wedge \iEnd}
\Bigabs{\X{t} - e^{-2t}\sumk k\nk}  = \op(n\tgan^2).
\end{equation}
\end{lemma}

\begin{proof}
 When in a state with $\xI \ge 1$ free infective half-edges, and thus $x \ge
 2$ \fhe{}s in total,  each free infective half-edge pairs off at rate
 $(x-1)/\xI$, and so the number of \fhe{}s  decreases by $2$ at rate
 $x-1$.
We modify the process so that pairs of \fhe{}s still disappear at rate $x-1$
when there are no more free infective half-edges (as long as $x\ge2$).
Let $\tilde{X}_t$ be the number of free half-edges at time $t$ in the modified process. Then it suffices to prove that
$$
\sup_{t \le  \tgan t_0} \Bigabs{\tilde{X}_t - e^{-2t}\sumk k\nk} = \op(n\tgan^2).$$

%Then $(\tilde{\X{t}}; t \le \iEnd)$ has the same law as $(\X{t}; t \le \iEnd)$,
%and so it suffices to prove the statement for $\tilde{\X{t}}$.

Now, $\tilde{\X{t}}-1$ is a linear death chain starting from $\sumk k\nk - 1$, and taking jumps from state $j$ to $j -2$ at rate $j$. By \cite[Lemma  6.1]{JansonLuczakGiantCpt}, with $d = 2$, $\gamma = 1$,
and $x = \sumk k\nk - 1$,
\[
 \E\sup_{t \le \tgan t_0}
\Bigabs{ (\tilde{\X{t}}-1) - e^{-2t}\Bigpar{\sumk k\nk - 1} }^2 \le 16 (e^{2\tgan t_0} - 1)\sumk k\nk  + 32.
\]

But $\sumk k\nk = O(n)$ by \ref{d:2ndmoment}, $\tgan t_0 = o(1)$ and
$n\tgan \ge n\gan\to \infty$ by \ref{d:alphan}, so the \rhs{} is
$O(n\tgan)$, and so
$\sup_{t \le  \tgan t_0} \bigabs{\tilde{X}_t - e^{-2t}\sumk k\nk} = \Op
\bigpar{\sqrt{n \tgan}} = \op(n\tgan^2)$, using \ref{d:alphan}.
% Finally use $\abs{ \X{t} - 2me^{-2t} } \le 1 + \abs{\X{t}' - (2m-1)e^{-2t}}$.
\end{proof}

\begin{proof}[Proof of \refT{t:conc}]
We start by proving  \eqref{e:XSLLN}, and, as remarked after the statement of \refT{t:conc}, it is enough to prove that
\begin{equation}
\label{e:XSLLNm}
 \sup_{t\le \tgan t_0 } \Bigabs{\sumk k\Svkm{t} - \fXSn(t)} = \op(n\tgan^2).
\end{equation}
Now, for each $k$, $(\Svkm{t})$ is a linear death chain starting from $\nSk$ and decreasing by 1
at rate $kx$ when in state $x$, and so
\begin{equation}\label{e:semimgdecomp0}
\Svkm{t} = \nSk - k \int_0^t \Svkm{u}du + \tilde{M}_{t}(k),
\end{equation}
where $\tilde{M}(k)=(\tilde{M}_t(k))$ is a zero-mean martingale. It follows that
\begin{align}
\sumk k \Svkm{t} - \fXSn(t)  & = \sumk k (\Svkm{t} - \nSk e^{-kt})  \nonumber \\
&= - \sumk k^2 \int_0^t (\Svkm{u} - \nSk e^{-ku})du + \tilde M_t \nonumber \\
& = -\int_0^t W_u du + \tilde M_t,\label{e:semimgdecomp}
\end{align}
where $\tilde M_t = \sum_{k} k \tilde{M}_{t}(k)$ defines a zero-mean martingale $\tilde{M}= (\tilde{M}_t)$.

Since $\Svkm{t}$ and $\tilde{S}_t(j)$ with $k \not = j$ never jump
simultaneously, we have $[\tilde{M}(k),\tilde{M}(j)]=0$,
where $[\cdot,\cdot]$ is the quadratic covariation, see e.g.\
\cite[Theorem 26.6]{kallenberg}.
Since also each
jump of  $\Svkm{t}$ is by $-1$,
the quadratic variation $[\tilde M]_t:=[\tilde M,\tilde M]_t$
is
\begin{align*}
[\tilde M]_t & = \sumk k^2 [M(k)]_t
 = \sumk k^2 \sum_{u \le t} (\Delta \Svkm{u})^2 \\
& = - \sumk k^2 \sum_{u \le t} (\Delta \Svkm{u})
 = \sumk k^2 (\nSk - \Svkm{t}) \\
& = \sumk k^2 (\nSk e^{-kt} - \Svkm{t} + \nSk(1 - e^{-kt})) \\
& \le \Bigabs{\sumk k^2 (\Svkm{t} - \nSk e^{-kt})} + t\sumk k^3 \nSk \\
& = \abs{W_t} + O(tn),
\end{align*}
again using \eqref{k3}.
%since $\sumk k^3 \nSk =  O(n)$ by \ref{d:cubeUI}.
By \refL{l:susdist2ndmomentBD},
$\E\sup_{t \le \tgan t_0}\abs{W_t} = o(n\tgan)$, so
%(more precisely, \eqref{e:cheekilyignoreiEnd}), and so
\begin{equation}\label{e:EQVbd}
\E[\tilde M]_{\tgan t_0}  = O(n \tgan),
\end{equation}
and so, using the Burkholder--Davis--Gundy inequalities
\cite[Theorem 26.12]{kallenberg},
%\begin{equation}
%\sup_{t \le \tgan t_0} \tilde M_t^2 = \Op(n \tgan).
%\end{equation}
%It follows that
$\sup_{t \le \tgan t_0} \abs{\tilde M_t} = \Op(\sqrt{n\tgan})$. Hence
by \eqref{e:semimgdecomp} and \refL{l:susdist2ndmomentBD}, uniformly in $t \le \tgan t_0$,
\begin{align*}
\Bigabs{\sumk k (\Svkm{t} - \nSk e^{-kt})} & \le \Bigabs{\int_0^t W_u du} + \Op(\sqrt{n\tgan}) \\
& \le \tgan t_0 \sup_{t \le \tgan t_0}  \abs{W_t} + \Op(\sqrt{n\tgan}) \\
& = \op(n \tgan^2),
\end{align*}
using again \ref{d:alphan}.
This establishes \eqref{e:XSLLN}.

%\end{proof}

%\begin{proof}[Proof of \refT{t:conc}; \eqref{e:SLLN}]
Next we prove \eqref{e:SLLN}.
By \cite[Lemma  6.1]{JansonLuczakGiantCpt}
with $d = 1$ and $\gamma = k$ and $x = \nSk$,
for $k \le \tgan^{-1}$,
\begin{equation}\label{e:JLpdcest}
 \E \sup_{t \le \tgan t_0} \abs{\Svkm{t} - \nSk e^{-kt}}^2 \le 4 ( e^{k\tgan t_0} - 1) \nSk \le 4 k\tgan t_0 e^{t_0} \nSk,
\end{equation}
where the last step uses the simple inequality $e^x-1\le xe^x$ for $x\ge0$.
For $k > \tgan^{-1}$, we use the trivial bound
$\abs{\Svkm{t} - \nSk e^{-kt}} \le \nSk$.
Using Jensen's inequality and then the Cauchy--Schwarz inequality, as well
as \eqref{k3} %\ref{d:cubeUI}, which implies that $\sumk k^3 \nSk =O(n)$,
and \ref{d:alphan},
\begin{align*}
 \E \sumk \sup_{t \le \tgan t_0} \abs{\Svkm{t} - \nSk e^{-kt}}
& \le \sum_{1\le k \le \tgan^{-1}} (4 k\tgan t_0 e^{t_0} \nSk)^{1/2}
+ \sum_{k > \tgan^{-1}} \nSk \\
 & \le 2(\tgan t_0 e^{t_0})^{1/2} \left(\sumki k^{-2} \sumki k^{3} \nSk \right)^{1/2}  + \tgan^3 \sumki k^3 \nSk \\
 & = O( \sqrt{n\tgan}) + O( n\tgan^3)
  = o(n\tgan^2),
\end{align*}
which yields \eqref{e:SLLNx} and thus \eqref{e:SLLN}.

\medskip

We now prove \eqref{e:XRLLN}.
The number of free recovered half-edges changes when either an infective
vertex recovers or
a free infective half-edge pairs with a free recovered half-edge. In the time-changed process, when in a state with $\xI$ free infective half-edges and $x$ free half-edges, infective vertices recover at rate $\rho_n (x-1)/\beta_n \xI$.
%
%Suppose there are $i_k$ infective vertices of degree $k \ge 0$, $\xI = \sumk k i_k \ge 1$ infective \fhe{}s, $\xR$ recovered \fhe{}s and $x$ \fhe{}s in total.
%Each vertex recovers at rate
%and so the total rate for a degree $k$ vertex to recover is $i_k \rho_n (x-1)/\beta_n \xI$.  This causes the number of recovered \fhe{}s to increment by $k$.
Also,
each free recovered half-edge is chosen to be paired at rate 1, and thus
the number of recovered \fhe{}s decreases
by 1 at rate $\xR$. Hence,
for any $t\ge0$,
\begin{equation}
\label{e:dXRtIntegrated}
\XR{t \wedge \iEnd} = \XR{0} - \int_0^{t \wedge \iEnd} \XR{s}ds + \frac{\rho_n}{\beta_n} \int_0^{t \wedge \iEnd}(\X{s} - 1) ds + M_{\mathrm{R}, t \wedge \iEnd},
\end{equation}
where $M_{\mathrm{R}} = (M_{\mathrm{R},t})$ is a zero-mean martingale.

On the other hand, differentiating \eqref{e:fXRnt} reveals that
\begin{equation}
\fXRn'(t) = -\fXRn(t) + \frac{\rho_n}{\beta_n}e^{-2t} \sumk k\nk.
\end{equation}
Hence, subtracting the integral of that expression from \eqref{e:dXRtIntegrated}, and recalling that $\XR{0}=0$,
%and then taking the absolute value inside the integral,
\begin{multline*}
\Bigabs{\XR{t\wedge\iEnd} -\fXRn(t\wedge\iEnd)}  \le
\int_0^{t \wedge \iEnd}\abs{\XR{s\wedge\iEnd} -\fXRn(s\wedge\iEnd)}ds \\
 \quad + \frac{\rho_n}{\beta_n}\int_0^{t \wedge \iEnd}
\Bigabs{\X{s} - 1 - e^{-2s} \sumk k\nk}ds
  + \abs{M_{\mathrm{R}, t\wedge\iEnd}}.
\end{multline*}
Then Gronwall's inequality yields
\begin{multline}\label{ika}
\sup_{t \le  \tgan t_0\wedge \iEnd} \abs{\XR{t} - \fXRn(t)}
 \le e^{\tgan t_0} \tgan t_0 \frac{\rho_n}{\beta_n} \left( \sup_{t \le \tgan
   t_0\wedge\iEnd} \Bigabs{\X{t} - e^{-2t} \sumk k\nk} + 1\right)  \\
\qquad +  e^{\tgan t_0} \sup_{t \le \tgan t_0\wedge\iEnd} \abs{M_{\mathrm{R},t}}.
\end{multline}
Since $\rho_n/\beta_n$ is bounded and $\tgan \to 0$,  the first term on the
\rhs{}
is $\op(n\tgan^2)$, by \eqref{e:XLLN}. It remains to show that the same is true of
the martingale term.

%Consider the quadratic variation,
%\begin{equation}\label{e:DeltaMSumSquares}
% [M_\mathrm{R}]_t = \sum_{s \le t} (\Delta M_{\mathrm{R},s})^2.
%\end{equation}
%(We may ignore $\iEnd$, as it is understood no jumps occur after then).
%The possible jumps are $\Delta M_s = -1$ (if a recovered \fhe{} pairs off)
%or $\Delta M_s = k$ (if an infective vertex with $k$ \fhe{}s recovers).
Note that $X_{R,t}$ jumps by $-1$ when a free recovered half-edge is paired with a free infective half-edge, and it jumps by $+k$ when an infective vertex with $k$ free half-edges recovers.
%
%A recovered \fhe{} can only pair with an infective \fhe{}.  The latter was either
%initially infective or was initially susceptible and then became infective.  There are
%\begin{equation}
%\label{e:DeltaMPairOffBd}
% \sumk k \nIk + \sumk k(\nSk - \Svk{t})
%\end{equation}
%such \fhe{}s by time $t$.
%
%In a similar vein, only infective vertices can recover.  In particular,
Also, each recovered half-edge or vertex was either initially infective or
was initially susceptible and then became infected prior to recovery.
%The number of \fhe{}s attached
%to a recovering vertex is at the most the number of \fhe{}s that the vertex
%had initially.
%Thus, the contribution to \eqref{e:DeltaMSumSquares} is at most
%\begin{equation}
%\label{e:DeltaMRecoverBd}
% \sumk k^2 \nIk + \sumk k^2(\nSk - \Sv{t}(k)).
%\end{equation}
Hence
\begin{align*}
\E [M_\mathrm{R}]_{\tgan t_0 \wedge \iEnd} & = \E [\sum_{s \le \tgan t_0 \wedge \iEnd} (\Delta M_{\mathrm{R},s})^2]\\
& \le
\E \sumk k \bigpar{\nIk + \nSk - \Sv{\tgan t_0 \wedge \iEnd}(k)}
+ \E \sumk k^2 \bigpar{\nIk + \nSk - \Sv{\tgan t_0 \wedge \iEnd}(k)} \\
& \le 2 \sumk k^2 \nIk
+ 2\E\sumk k^2\bigpar{\nSk - \Sv{\tgan t_0 \wedge \iEnd}(k)}\\
& = 2 \sumk k^2 \nIk
+ 2\E\sumk k^2\bigpar{\nSk(1  - e^{-k(\tgan t_0 \wedge \iEnd)})
+\nSk e^{-k(\tgan t_0 \wedge \iEnd)}
- \Sv{\tgan t_0 \wedge \iEnd}(k)}\\
& \le 2 \sumk k^2 \nIk
+2\tgan t_0 \sumk k^3 \nSk
+ 2\E\sumk k^2\bigpar{\nSk e^{-k(\tgan t_0 \wedge \iEnd)}
- \Sv{\tgan t_0 \wedge \iEnd}(k)}\\
& = o\bigpar{n^2\tgan^4}+ O(n\tgan)+o(n\tgan)
 = o\bigpar{n^2\tgan^4}
\end{align*}
by \eqref{k2I}, \eqref{k3},  \refL{l:susdist2ndmomentBD}
and $n\tgan^3\ge n\ga_n^3\to\infty$, see \ref{d:alphan}.
Then, by
the Burkholder--Davis--Gundy inequalities,
$\sup_{t \le \tgan t_0\wedge\iEnd} \abs{M_{\mathrm{R},t}}
= \op(n\tgan^2)$,
and \eqref{e:XRLLN} follows by \eqref{ika}.

\medskip

Finally, \eqref{e:XILLN}
follows from \eqref{e:XSLLN}, \eqref{e:XRLLN}, \eqref{e:XLLN},
and the fact that $\XI{t} = \X{t} - \XS{t} - \XR{t}$.
\end{proof}

\subsection{Duration of the time-changed epidemic}\label{s:duration}
We stated \refT{t:conc} using a rather arbitrary $\tgan$, but from now
on we fix it as follows. We distinguish between the cases $\nu<\infty$ and
$\nu=\infty$, and introduce some further notation:

If $0\le\nu<\infty$,
define
\begin{align}
  \tgan&:=\ga_n, \label{m1a}\\
  f(t)&:=\nu+t-\tfrac{\gliii}2t^2, \label{m1f}\\
\xl&:=\xxl. \label{m1x}
\end{align}
If $\nu=\infty$,
define instead
\begin{align}
  \tgan&:=\Bigpar{\sumk k\nIk/n}\qq, \label{m2a}\\
  f(t)&:=1-\tfrac{\gliii}2t^2, \label{m2f}\\
 \xl&:=\sqrt{2/\gliii}.\label{m2x}
\end{align}
Note that in both cases, $\xl$ is the unique positive root of $f$, and that
$f(t)>0$ on $(0,\xl)$ and $f(t)<0$ on $(\xl,\infty)$;
we have $f(0)=0$ if $\nu=0$ but $f(0)>0$ if $\nu>0$.
Note further that in the case $\nu=\infty$, $\tgan/\ga_n\to\infty$ by
\eqref{e:HI}; in particular, $\tgan\ge\ga_n$ except possibly for some small
$n$ that we will ignore. Moreover, $\tgan\to0$ by
\ref{d:alphan} ($\nu<\infty$) or
\ref{d:initinfective} ($\nu=\infty$).

Next, if $\nu<\infty$, then, by \eqref{e:HI},
\begin{equation}
  \sumk k\nIk = O\bigpar{\nS\ga_n^2}= O\bigpar{n\tgan^2},
\end{equation}
and if $\nu=\infty$, then by \eqref{m2a},
\begin{equation}
  \sumk k\nIk =  n\tgan^2.
\end{equation}
Hence, in both cases,
\begin{equation}\label{mk}
%\fXIn(0)=
  \sumk k\nIk =  O\bigpar{n\tgan^2}.
\end{equation}
Furthermore, if $\nu=0$ then \eqref{e:HI} yields
$\sum_k k\nIk = o\bigpar{n\ga_n^2}$ and thus
\begin{equation}
  \sumk k^2\nIk
\le
 \Bigpar{ \sumk k\nIk}^2
= o\bigpar{n^2\tgan^4},
\end{equation}
and if $0<\nu\le\infty$ then \eqref{RdImax} and \eqref{mk} imply
\begin{equation}
  \sumk k^2\nIk
\le \dImax   \sumk k\nIk
=o \Bigpar{ \sumk k\nIk}^2
= o\bigpar{n^2\tgan^4}.
\end{equation}
Hence, \eqref{k2I} holds in all cases.

We have verified that our choice of $\tgan$ satisfies the conditions of
\refT{t:conc}, so
\refT{t:conc} applies.
We use this to show a more explicit limit result for $\XI{t}$.
%(Recall our simplifying assumption that $\nR = 0$.)
\begin{lemma}\label{LXI}
  For any fixed $t_0$,
\begin{equation}\label{XIconc}
\sup_{t \le t_0 \wedge (\iEnd/\tgan)}
\Bigabs{\frac{\XI{\tgan t}}{n\tgan^2} - f(t)} \pto 0.
\end{equation}
\end{lemma}

\begin{proof}
The idea is to combine \refT{t:conc} with a Taylor expansion of $\fXIn(t)$
around zero.

%and so the %fraction of susceptible vertices satisfies $\nS/n \to 1$.
The first three derivatives of $\fXIn(t)$ are
\begin{align}
\label{e:dHInt}
 \fXIn'(t) &= -2 e^{-2t} \sumk k\nk + \sumk k^2 \nSk e^{-kt} +
  \frac{\rho_n}{\beta_n}e^{-t}(1 - 2e^{-t})\sumk k\nk,
\\%\end{equation}
%\begin{equation}
\label{e:ddHInt}
 \fXIn''(t) &= 4 e^{-2t} \sumk k\nk - \sumk k^3 \nSk e^{-kt} -
  \frac{\rho_n}{\beta_n}e^{-t}(1 - 4e^{-t})\sumk k\nk,
\\%\end{equation}
%\begin{equation}
\label{e:dddHInt}
 \fXIn'''(t) &= -8 e^{-2t} \sumk k\nk + \sumk k^4 \nSk e^{-kt} +
  \frac{\rho_n}{\beta_n}e^{-t}(1 - 8e^{-t})\sumk k\nk.
\end{align}
Hence, using \eqref{e:alphan}, \eqref{nns}, \eqref{mk} and $\tgan\to0$,
\begin{align}
 \fXIn'(0) = -2  \sumk k\nk + \sumk k^2 \nSk
  -\frac{\rho_n}{\beta_n}\sumk k\nk
 = \nS\alpha_n - \sumk k\nIk = n\alpha_n + o(n\tgan),\label{e:dHInt0}
 \end{align}
and similarly, using also \eqref{gliii2},
 \begin{align}
 \fXIn''(0) &= 4 \sumk k\nk - \sumk k^3 \nSk  +
  3\frac{\rho_n}{\beta_n}\sumk k\nk \nonumber \\
    & = 3(1 + \rho_n/\beta_n) \sumk k\nk + \sumk k\nSk  - \sumk k^3 \nSk + \sumk k \nIk \nonumber \\
  & = 3(1 + \rho_n/\beta_n) \sumk k\nk - \sumk k(k-1)(k-2 + 3)\nSk + \sumk k\nIk \nonumber \\
  & = -3\nS\alpha_n - \sumk k(k-1)(k-2)\nSk + O(n\tgan^2)
=-n\gliii+o(n)
\label{e:ddHInt0}.
  \end{align}
%We now bound the third derivative appropriately, in order to control the error
%in the Taylor expansion.
%The ratio $\rho_n/\beta_n$ is bounded as before, and $\sumk k \nk = O(n)$ by \ref{d:2ndmoment}.
Also, for $ t\ge 0$,
\begin{align*}
\abs{\fXIn'''(t)}
& \le (8 + 7\rho_n/\beta_n)\sumk k \nk
 +  \sumk k^4 \nSk e^{-kt} = O(n)  +  \sumk k^4 \nSk e^{-kt}.
\end{align*}
Hence,  for any $M \ge1$,
%\marginal{Alternative: Instead of splitting at an $M$,use $\sum_{k} k^3 \nSk \xpar{1- e^{-k\tgan t_0}}/n \to0$ by %dominated convergence using uniform summability}
\begin{align*}
\int_0^{\tgan t_0}\abs{\fXIn'''(t)} \dd t
 & \le O(n\tgan) + \tgan t_0 \sum_{k \le M}  M^4\nSk
+ \sum_{k > M} k^3 \nSk \bigpar{1- e^{-k\tgan t_0}} \\
 & \le O(M^4 n\tgan) + \sum_{k > M} k^3 \nSk
 = o(n) + \sum_{k > M} k^3 \nSk.
\end{align*}
Letting $M \to \infty$ slowly (so that $M^4 \tgan =o(1)$), and using
\ref{d:cubeUI}, we obtain
\begin{equation}
\label{e:dddHIntBD}
\lim_{\ntoo} \frac{1}n \int_0^{\tgan t_0}\abs{\fXIn'''(t)} \dd t = 0.
\end{equation}

Now, by a Taylor expansion, for $t\ge0$,
\begin{equation}
\fXIn(\tgan t)
= \fXIn(0) + \fXIn'(0)\tgan t + \tfrac12\fXIn''(0)(\tgan t)^2
+ \tfrac12\int_0^{\tgan t}(\tgan t-u)^2\fXIn'''(u)\dd u,
\end{equation}
and hence,
using  $\fXIn(0)=\sum_k k\nIk$,
\eqref{e:dHInt0}, \eqref{e:ddHInt0} and \eqref{e:dddHIntBD},
uniformly in $t \le t_0$,
\begin{equation}\label{HITaylor}
\fXIn(\tgan t)
= \sumk k\nIk + n\alpha_n\tgan t - \tfrac12\tgan^2 t^2 n\gliii
+ o(n\tgan^2).
\end{equation}

If $\nu<\infty$, then $\tgan=\ga_n$, and
\eqref{HITaylor} yields by \eqref{e:HI} and \eqref{nns}
\begin{equation}
  \begin{split}
\frac{\fXIn(\tgan t)}{n\tgan^2} =
\nu + t - \tfrac12 t^2 \gliii + o(1)
=f(t)+o(1).
  \end{split}
\end{equation}
If $\nu=\infty$, then \eqref{HITaylor} yields similarly by \eqref{m2a} and
$\ga_n=o(\tgan)$,
\begin{equation}
  \begin{split}
\frac{\fXIn(\tgan t)}{n\tgan^2} =
1 - \tfrac12 t^2 \gliii + o(1)
=f(t)+o(1).
  \end{split}
\end{equation}
Consequently, in both cases
$%\begin{equation}
\xfrac{\fXIn(\tgan t)}{n\tgan^2}
=f(t)+o(1),
$ %\end{equation}
uniformly for $0\le t\le t_0$,
and the result follows by combining this and
\eqref{e:XILLN}
from \refT{t:conc}.
\end{proof}

We can now find (asymptotically) the duration $\iEnd$, except that when
$\nu=0$, we cannot yet say whether the epidemic is very small or rather large.

\begin{lemma}\label{l:iEnd}
\begin{enumerate}[\upshape(i)]
\item
If\/ $0<\nu\le\infty$, then
$\iEnd/\tgan\pto\xl$.
\item \label{liend0}
If $\nu=0$, then for
every $\eps > 0$,  \whp{}, either
 \begin{enumerate}[\upshape(a)]
 \item $0 \le \iEnd/\alpha_n < \eps$, or
 \item $\abs{\iEnd/\alpha_n - \xl}  < \eps$.
 \end{enumerate}
\end{enumerate}
%\marginal{Does a corresponding result hold away from criticality and can we use it to analyse
%duration of the original process?}

In particular, in both cases, \whp{} $\iEnd\le 2\xl\tgan$.
\end{lemma}

\begin{proof}
  Take $t_0=2\xl$. Then $f(t_0)<0$, so \eqref{XIconc} implies that
$\P(\iEnd/\tgan\ge t_0)\to0$, \ie,
$\iEnd< t_0 \tgan$ \whp.
Consequently, we may \whp{} take $t=\iEnd/\tgan$ in \eqref{XIconc} and
conclude
$\bigabs{\xfrac{\XI{\iEnd}}{n\tgan^2} - f(\iEnd/\tgan)} \pto 0$.
Since $\XI{\iEnd}=0$ by definition, this says $f(\iEnd/\tgan) \pto 0$.

Consider $f(t)$  for $t\in[0,\infty)$.
If $\nu>0$, then $f(t)$ has  a unique zero at $\xl$, and is bounded away
from 0 outside every neighbourhood of $\xl$; hence $\iEnd/\tgan\pto\xl$ follows.
If $\nu=0$, $f(t)=0$ both for $t=0$ and $t=\xl$, and \eqref{liend0} follows.
\end{proof}

\begin{remark}\label{Rt1}
\refL{l:iEnd} shows that taking $t_0 \= 2\xl$ in \refT{t:conc}, we have \whp{}
  $\tgan t_0\wedge \iEnd=\iEnd$, and thus, for $\tgan$ as above,
\refT{t:conc} holds also
with  the suprema taken over all $t\le\iEnd$.
\end{remark}

\subsection{Final size}

\begin{proof}[Proof of \refT{t:criticalSIR}]
Recall that $\fs_k \= \nSk - \Svk{\iEnd}$ is the number of susceptibles of
degree $k$ that ever become infected, and $\fs=\sum_k\fs_k$.
For each $k \in {\mathbb Z}^+$,
\begin{align}\label{koo}
\left|\frac{\fs_k}{n\tgan} - k p_k \frac{\iEnd}{\tgan} \right|
& \le
 \left|\frac{\nSk(1- e^{-k \iEnd})}{n\tgan} - k p_k \frac{\iEnd}{\tgan}\right|
 +  \left|\frac{\Svk{\iEnd} - \nSk e^{-k \iEnd}}{n\tgan}\right|.
% \left|\frac{\fs_k - \nSk e^{-k \iEnd}}{n\alpha_n} - k p_k \iEnd/\alpha_n\right|
% \left|\frac{\nS - \fSvn(\iEnd) } {n\alpha_n} - \lambda \frac{\iEnd \nS}{n \alpha_n} \right|
%  + \left| \frac{ \Sv{\iEnd} - \fSvn(\iEnd)}{n\alpha_n} \right| \\
%  & = \frac{\nS}{n}\left|\frac{\nS - \fSvn(\iEnd)} {\nS\alpha_n} - \lambda \frac{\iEnd}{\alpha_n} \right|
%  + \op(1) \\
%  & \pto 0.
\end{align}
Since $\abs{1 - e^{-y}-y} \le y^2$ for all $y \ge 0$, and using
\ref{d:asympsuscdist}, \ref{d:cubeUI} and \eqref{nns},
\begin{align*}
% \nS - \fSvn(\tgan t) & =
\sumk \Bigabs{\frac{(1 - e^{-k \tgan t})\nSk}{n\tgan} - kp_k t}
& \le t \sumk k\abs{p_k - \nSk/n} + \tgan t^2\sumk \frac{\nSk k^2}{n}
 \to 0,
\end{align*}
uniformly in $t \le t_0 \=  2\xl$. Since $\iEnd/\tgan\le t_0$ \whp{} by
\refL{l:iEnd}, it follows that
$$\sumk \left|\frac{\nSk(1- e^{-k \iEnd})}{n\tgan} - k p_k
\frac{\iEnd}{\tgan}\right| = \op (1).
$$
Further, by \eqref{e:SLLN} in \refT{t:conc} and Lemma~\ref{l:iEnd}, see
\refR{Rt1},
%shows the
%total deviations of the $\Svk{\iEnd}$ around $\nSk e^{-k \iEnd}$, summed over $k$, is $\op(n\alpha_n^2)$.
$$\sumk \left|\frac{\Svk{\iEnd} - \nSk e^{-k \iEnd}}{n\tgan}\right| = \op (\tgan) = \op (1).$$
It follows by \eqref{koo} that
\begin{equation}\label{e:Zkdev}
 \sumk \left|\frac{\fs_k}{n\tgan} - k p_k \frac{\iEnd}{\tgan} \right| = \op(1)
\end{equation}
and, in particular,
\begin{equation}\label{Zdev}
 \left|\frac{\fs}{n\tgan} - \gl \frac{\iEnd}{\tgan} \right| =
 \left| \sumk\lrpar{\frac{\fs_k}{n\tgan} - k p_k \frac{\iEnd}{\tgan}} \right|
= \op(1).
\end{equation}
The estimate \eqref{Zdev} and \refL{l:iEnd}
yield the conclusions
\ref{tsizei}--\ref{tsizeiii} by \eqref{nns} and the definitions
of $\tgan$ and $\xl$ in \eqref{m1a}--\eqref{m2x}.
For \ref{tsizei}, when $\nu=0$, we first obtain that
if $\eps_n=\eps>0$ is fixed but arbitrary,
then the conclusion holds \whp, and it is easy
to see that this implies the same for some sequence $\eps_n\to0$.
(Note also that if $\nu=0$, then $\xl=2/\gliii$.)

Furthermore,
combining \eqref{e:Zkdev} and \eqref{Zdev}, we obtain
\begin{equation}\label{Zkdev}
\sumk \left|\frac{\fs_k}{n\tgan} - \frac{k p_k}{\gl} \frac{\fs}{n\tgan} \right|
 = \op(1).
\end{equation}
We have shown that, except in the case \ref{tsizei}\ref{tsizei:small},
there exists $\eps>0$ such that \whp{}
$\fs/n\tgan\ge\eps$; then \eqref{t2ab} follows from \eqref{Zkdev}.
\end{proof}

\section{Proof of \refT{t:takeoff}}\label{Spftakeoff}
We continue to use the simplifying assumptions in \refS{s:simplify}.
We consider the epidemic in the original time scale and construct it
from independent exponential random variables.
At time $t=0$, we allocate each of the $\nI$ initially infective vertices an
$\Exp(\rho_n)$ recovery time.  We also give each free infective half-edge at
time 0 an $\Exp(\beta_n)$
pairing time.  If the pairing time for a \fihe{} is less than the recovery
time of its parent vertex, then we colour that \fhe{} red.  Otherwise, we
colour it black.  We now wait until the first recovery or pairing time.
At a recovery time, we change the
status of the corresponding vertex to recovered.  At a pairing time of a
red \fhe{}, we choose another \fhe{} uniformly at random.  If the
chosen \fhe{} belongs to a susceptible vertex then that vertex becomes
infective, is
given an $\Exp(\rho_n)$ recovery time, and its
remaining \fhe{}s are given independent $\Exp(\beta_n)$ pairing times.
Then, as above,
we colour red any \fhe{} with pairing time less than recovery time,
and colour black all other free half-edges at the chosen vertex.
The process continues in this
fashion until
no red \fhe{}s remain.  Note that we do nothing at the pairing time
of a black \fhe{}, since it is no longer infective, and so black
\fhe{}s behave like recovered \fhe{}s.  Also, a red \fhe{} will
definitely initiate a pairing event at some point (provided it has not been
chosen by another red \fhe{} first).
However, ignoring the colourings we obtain the same process as
before.

Let $Z_t$ be the number of red \fhe{}s at time $t \ge 0$. Note that $Z_t$
changes only at pairing events, but not at recovery times. (The point of the
colouring is to anticipate the recoveries, which then can be ignored.)
Further, let $\Zm:=Z_{T_m}$, where $T_m$ is the time of the $m$:th pairing
event (and $T_0:=0$), and let $\zeta_m:=\Delta \Zm:=\Zm-\Zx{m-1}$.
(Note that our processes are all right continuous, so $\Zm$ is the number
of red \fhe{}s immediately after the $m$-th pairing has occurred and we have
coloured any new infective \fhe{}s.)
Thus the process stops at $T_{\mx}$, where $\mx:=\min\set{m\ge0:\Zm=0}$.
(This is not exactly the same stopping condition as used earlier, but the
difference does not matter; there may
still be some infective half-edges, but they are black and will recover
before infecting any more vertex.)
Let $\cF_m =\cF(T_m)$
be the corresponding discrete-time filtration generated by the
coloured SIR process
up to time $T_m$.

We keep the same notation as before
for the \fhe{} counts (so the total number of free infective half-edges,
whether red or black, is $\XI{t} \ge Z_t$, for example), and write
again $S_t(k)$
for the number of susceptible vertices with $k$ \fhe{s} at time $t \ge 0$.
Furthermore, define
\begin{equation}\label{pin0}
  \pi_n:=\frac{\gb_n}{\gb_n+\rho_n},
\end{equation}
the probability that a given free infective half-edge is coloured red.
Note that $c\le\pi_n\le1$ for some $c>0$ by \eqref{rhobeta}.

We begin by showing that a substantial fraction of the initially infective
half-edges are red.
Recall that $\XI0=\sumk k \nIk$ is the total degree of the initially
infective vertices and that $\dImax$ is the maximum degree among these
vertices.
%see \eqref{RdImax},

\begin{lemma}\label{l:alphaZ0}
Suppose that\/ $\XI0 \to \infty$.
\begin{romenumerate}
\item \label{l:alphaZ01}
If $\dImax=o(\XI0)$, then
$Z_0=\pi_n \XI0\bigpar{1+\op(1)}$.

\item \label{l:alphaZ02}
More generally, for any $\dImax$,
we have
%$\XI0/Z_0=\Op(1)$, \ie,
\begin{equation}\label{e:Z0Plb}
\lim_{\delta \to 0} \limsup_{n\to\infty}
\Prob\Bigpar{Z_0 \le \delta \XI0} = 0.
\end{equation}
\end{romenumerate}
\end{lemma}

\begin{proof}
We enumerate all initially infective vertices as $i = 1, \ldots, \nI$, and let
$\dIi$ be the degree of vertex $i$, so that $\XI0 = \sum_{i = 1}^{\nI} \dIi$.
We also let $Z_{0,i}$ be the number of red \fhe{}s
at vertex $i$, so  $Z_0 = \sum_{i = 1}^{\nI} Z_{0,i}$, where the $Z_{0,i}$
are independent, with $\E Z_{0,i} = \dIi\pi_n $
and $Z_{0,i} \le \dIi$.
It follows that $\E Z_0=\sum_{i = 1}^{\nI}\E Z_{0,i} =\pi_n\XI0$ and
\begin{equation}\label{varZ0}
\Var Z_0 = \sum_{i = 1}^{\nI} \Var Z_{0,i} \le \sum_{i = 1}^{\nI} \dIi^2
 \le \dImax \XI0.
\end{equation}

\pfitemref{l:alphaZ01}
If $\dImax=o(\XI0)$, then \eqref{varZ0} yields
$\Var Z_0=o(\XI0^2)=o((\E Z_0)^2)$, and thus Chebyshev's inequality
yields $Z_0=\E Z_0(1+\op(1))$.

\pfitemref{l:alphaZ02}
Take any $\delta > 0$ with $\gd<\frac12\min_n \pi_n$.

We assume first that $\dImax \le \delta^{1/2} \XI0$.
Then \eqref{varZ0} and Chebyshev's inequality yield
\begin{equation}\label{hk3}
  \P(Z_0\le\gd\XI0)
\le \frac{\Var Z_0}{(\E Z_0-\gd\XI0)^2}
\le \frac{\dImax\XI0}{(\frac12\pi_n\XI0)^2}
\le 4\pi_n^{-2}\gd\qq
= 4\Bigpar{\frac{\rho_n+\gb_n}{\gb_n}}^{2}\gd\qq.
\end{equation}

Assume now instead that $\dImax \ge \delta^{1/2} \XI0$.
Fix one initially infective vertex of degree $\dImax$,
let $\Zzs$ be the number of red \fhe{}s at that vertex,
and let $R_\star$ be its recovery time.
Then $Z_0 \ge \Zzs$, and so $Z_0 \le \delta \XI0$ implies
that $\Zzs \le \delta^{1/2} \dImax$.  We have
\begin{align}
\Prob(\Zzs \le \delta^{1/2} \dImax) & = \Prob(\Zzs \le \delta^{1/2}\dImax, R_\star \le 4\delta^{1/2}/\beta_n ) +
\Prob(\Zzs \le \delta^{1/2} \dImax, R_\star > 4\delta^{1/2}/\beta_n )
\notag\\
& \le \Prob( R_\star \le 4\delta^{1/2}/\beta_n )
+ \Prob\bigpar{\Zzs \le \delta^{1/2}\dImax
\mid R_\star > 4\delta^{1/2}/\beta_n}.
\label{hk0}
\end{align}
Now,
\begin{equation}\label{latte}
\Prob( R_\star \le 4\delta^{1/2}/\beta_n )
= 1 - e^{-4\delta^{1/2}\rho_n/\beta_n}
\le 4\delta^{1/2}\rho_n/\beta_n.
\end{equation}
Also, conditional on $R_\star = r$, $\Zzs$ has a binomial distribution with
parameters $\dImax$ and $1-e^{-\beta_n r}$.  It follows that, conditional on
$R_\star > 4\delta^{1/2}/\beta_n$, $\Zzs$ stochastically
dominates a $\Bin(\dImax, 1 - e^{-4\delta^{1/2}} )$ random variable.
For $\delta$ small enough, % say $\delta< 2^{-12}$, we have
$1 - e^{-4\delta^{1/2}} \ge 2 \delta^{1/2}$, and so, by
Chebyshev's inequality,
\begin{align}\label{hk1}
\Prob\Bigpar{\Zzs \le \delta^{1/2} \dImax \mid R_\star >
  \frac{4\delta^{1/2}}{\beta_n}}
& \le \Prob\Bigpar{\Bin(\dImax, 2\delta^{1/2}) \le \delta^{1/2} \dImax }
\le \frac{2\gd\qq\dImax}{\gd\dImax^2}
\le \frac{2}{\delta \XI0}.
\end{align}

Combining  \eqref{hk3}, \eqref{hk0}, \eqref{latte} and \eqref{hk1}, we see
that if $\gd$ is
small, then in both cases
\begin{equation*}
\Prob \bigpar{Z_0 \le \delta \XI0 }
\le
 4\Bigpar{\frac{\rho_n+\gb_n}{\gb_n}}^{2}\gd\qq
+ \frac{2}{\delta \XI0},
\end{equation*}
and \eqref{e:Z0Plb} follows, since
$\XI0\to\infty$ and
$\rho_n/\gb_n=O(1)$ by \eqref{rhobeta}.
\end{proof}

\begin{lemma}\label{l:RWdrift}
Let $(W_m)_{m=0}^\infty$ be a process adapted to a filtration $(\cF_m)_{m =
  0}^\infty$, with $W_0 = 0$, and let $\tau\le\infty$ be a stopping time.
Suppose that the positive numbers $v,w > 0$ are such that
\begin{align}
\E[ \Delta W_{m+1}\mid\cF_m] &\ge v \quad\text{a.s.~on } \set{m<\tau},
\label{e:RW0}
\\
\E[(\Delta W_{m+1})^2] &\le w \label{e:RW2}
\end{align}
for every $m \ge 0$.
Then, for any $b>0$,
\begin{equation}\label{e:lRWbdgoal}
\Prob\left(\inf_{0\le m \le \tau} W_m \le -b\right) \le \frac{8 w}{bv}.
\end{equation}
\end{lemma}

\begin{proof}
Consider the Doob decomposition
\begin{equation}
W_m = M_m + A_m,
\end{equation}
where $A_m \= \sum_{l=1}^m \E[\Delta W_{l}\mid \cF_{l-1}]$ is predictable and
$M_m \= W_m - A_m$ is a martingale with respect to $(\cF_m)$.
By the assumption \eqref{e:RW0}, $A_m\ge mv$ \as{} when $m\le\tau$.
Furthermore,
\begin{align*}
\E[\gD M_m^2]  &= \E[(\Delta W_m - \E[\Delta W_m \mid\cF_{m-1}] )^2]
 = \E[(\Delta W_m)^2] - \E(\E[\Delta W_m \mid\cF_{m-1}])^2
 \le w.
\end{align*}
Thus, by Doob's inequality, for any $N\ge1$,
\begin{equation}
  \begin{split}
\Prob\Bigpar{\inf_{N\le m \le (2N) \wedge\tau} W_m \le -b}
\le
\Prob\Bigpar{\inf_{N\le m \le 2N} M_m \le -b-Nv}
\le \frac{\E[M_{2N}^2]}{(b+Nv)^2}
\le \frac{2N w}{(b+Nv)^2}.
  \end{split}
\end{equation}
Summing over all powers of 2, we obtain
\begin{equation*}
  \begin{split}
\Prob\Bigpar{\inf_{1\le m \le \tau} W_m \le -b}
&\le \sumk
\Prob\Bigpar{\inf_{2^k\le m \le 2^{k+1}\wedge\tau} W_m \le -b}
\le
\sumk \frac{2^{k+1} w}{(b+2^kv)^2}
\\&
\le \sum_{2^k \le b/v} \frac{2^{k+1} w}{b^2}
+ \sum_{2^k \ge b/v} \frac{2^{k+1} w}{2^{2k}v^2}
\le \frac{4(b/v)w}{b^2}+\frac{4(v/b)w}{v^2}=\frac{8w}{bv}.
\qedhere
  \end{split}
\end{equation*}
\end{proof}

\begin{proof}[Proof of Theorem~\ref{t:takeoff}\ref{takeoff2}.]
Let $\deps > 0$ be a small positive number chosen later. Define the discrete
stopping time $\mxx$
by
\begin{equation}
  \label{mxx}
\mxx:=\min\Bigcpar{m\ge0: \Zm=0 \quad\text{or}\quad
\sumk k^2 (\nSk -  S_{T_m}(k)) >  \deps n \alpha_n}.
\end{equation}
%Recall that $\X{t}$ is the total number of \fhe{s}
%and
Note that for $m<\mxx$,
the total number of \fhe{s} at time $T_m$ is
\begin{equation}%\label{mxx1}
  \begin{split}
\X{T_m}
\ge \sumk kS_{T_m}(k)
\ge \sumk k\nSk - \deps n\ga_n
\ge \sumk k\nk - \deps n\ga_n -o(n\ga_n),
  \end{split}
\end{equation}
since $\sumk k\nIk=o(\nS\ga_n^2)=o(n\ga_n)$ by \eqref{e:HI} and
\eqref{e:alphanasy}.
Similarly,
for $m<\mxx$,
\begin{equation}\label{mxxz}
 \Zm=Z_{T_m}
\le Z_0 + \sumk k (\nSk -  S_{T_m}(k))
\le Z_0 +  \deps n\alpha_n
\le \deps n\alpha_n + o(n\ga_n),
 \end{equation}
since $Z_0$ is bounded above by $\sumk k\nIk = o(\nS \alpha_n^2)
= o(n\alpha_n)$. % by (D4) and (D5).

At a pairing $m +1 \le \mxx$, a red free half-edge pairs with a free
susceptible half-edge, or with another red free half-edge, or with a black
half-edge.
In the first case, if the susceptible half-edge belongs to a vertex of
degree $k$, we get on the average $\pi_n (k-1)$ new red \fhe{s}; in the
second case we instead lose one red \fhe, in addition to the pairing
red \fhe{} that we always lose.
The probability of pairing with a susceptible half-edge belonging to a vertex
of degree $k$ is $kS_{T_m}(k)/X_{T_m}$ and the probability of pairing with
another red \fhe{} is $\Zm/X_{T_m}$.
Hence, for $m+1 \le \mxx$, using \eqref{mxx}--\eqref{mxxz},
\eqref{pin0} and the definition
\eqref{R0} of $\cR_0$,
\begin{align*}
\E[\Delta \Zx{m+1}\mid\cF_m]
&\ge
-1 + \pi_n %\frac{\beta_n}{\rho_n + \beta_n}
\frac{\sumk (k-1)k S_{T_m}(k)}{\sumk k\nk}
- \frac{\Zm}{\sumk k\nk - (\deps+o(1)) n \alpha_n}\\
& \ge -1 + \pi_n %\frac{\beta_n}{\rho_n + \beta_n}
\frac{\sumk (k-1)k\nSk-\deps n\ga_n}{\sumk k\nk}
-\frac{(\deps+o(1))n\ga_n}{\sumk k\nk- (\deps+o(1)) n \alpha_n}\\
& \ge -1  + \Rzero - O(\deps \alpha_n).
\end{align*}
Since $(\Rzero-1)\alpha_n^{-1}$ is bounded away from 0 by \eqref{aR0}
and \refR{Rrhobeta},
this shows that as long as $\deps$ is chosen small enough
there exists some $\cc > 0\ccdef\cczmi$ such that
if  $n$ is large and $m< \mxx$, then
\begin{equation}\label{ezmi}
\E[\Delta \Zmi\mid\cF_m] \ge \ccx \alpha_n.
\end{equation}
Furthermore,
noting that the number of red free half-edges may change by at most $k$ at a
jump if a red free half-edge pairs with a free susceptible half-edge at a
vertex of degree $k$, the expected square of any jump satisfies,
for $m<\mxx$,
\begin{align}\label{ezmi2}
\E[ (\Delta \Zmi)^2\mid \cF_m ]
\le 4 + \frac{\sumk k^3\nSk}{\sumk k\nk-(\deps+o(1))n\ga_n} \le \cc,
\ccdef\cczmii
\end{align}
for some $\ccx> 0$, uniformly in all large $n$, by assumption \ref{d:cubeUI}.

Let $W_m = \Zx{m\wedge\mxx}-Z_0$.
It follows from \eqref{ezmi}--\eqref{ezmi2} that \refL{l:RWdrift} applies
with $\tau=\mxx$, $v=\cczmi\ga_n$ and $w=\cczmii$.

Let $a=a_n> 0$ satisfy $a_n \to \infty$ and
$a_n =o\bigpar{\alpha_n \sumk k \nIk}$
as $n \to \infty$.
Then \refL{l:RWdrift} with $b=a_n/\ga_n$ yields
\begin{equation}%\label{e:lRWbdgoal}
\Prob\lrpar{\inf_{0\le m \le \mxx} W_m \le -a_n/\ga_n} \le
\frac{8 \ccx}{\cczmi a_n} =o(1)
\end{equation}
and thus $W_{\mxx}>-a_n/\ga_n$ \whp.
On the other hand, \refL{l:alphaZ0}\ref{l:alphaZ02}
implies that $\P(Z_0\le a_n/\ga_n)\to0$.
Consequently, $\Zx{\mxx}=W_{\mxx}+Z_0>0$ \whp.
By \eqref{mxx}, this means that \whp{}
$\sumk k^2 (\nSk -  S_{T_{\mxx}}(k)) >  \deps n \alpha_n$, and thus that,
for some time $t$ before the epidemic dies out,
$\sumk k^2 (\nSk -  S_{t}(k)) >  \deps n \alpha_n$.

The latter statement does not depend on the time-scale, so it holds for the
time-changed epidemic in \refS{s:acceleratedepidemic} too.
Thus, using the notation there, by the monotonicity of the number of
susceptibles,  \whp{}
\begin{equation}\label{paddington}
 \sumk k^2 (\nSk -  S_{\iEnd}(k)) >  \deps n \alpha_n.
\end{equation}

On the other hand, \refL{l:susdist2ndmomentBD} (with $\tgan=\ga_n$)
and \eqref{k3} yield, for any $\eps>0$,
\begin{equation}\label{padda}
  \begin{split}
\sup_{0 \le t \le \eps \alpha_n \wedge \iEnd }\sumk k^2(\nSk - \Svk{t})
%& =\sumk k^2(\nSk - \Svk{\eps\alpha_n  \wedge \iEnd}) \nonumber \\
& \le \sup_{t\le \eps \alpha_n \wedge \iEnd}
  \Bigabs{\sumk k^2(\Svk{t} - \nSk e^{-k t})}
+ \eps \alpha_n \sumk k^3 \nSk \\
& \le \op (n \alpha_n)+ \eps c_0 n \alpha_n
% \le 2 \eps  c_0 n \alpha_n
.
  \end{split}
\end{equation}

Consequently, if we first choose $\deps>0$ so small that \eqref{paddington}
holds, and then $\eps<\deps/c_0$, then \eqref{paddington}--\eqref{padda}
imply that \whp{} $\iEnd>\eps\ga_n$.
It follows by \refL{l:iEnd} that
$\iEnd/\alpha_n \pto \xl=2/\gliii$, and thus the result follows by \eqref{Zdev}.
\end{proof}

To study the cases \ref{takeoff1} and \ref{takeoff3}, we analyse the number of
red \fhe{s} more carefully.
Let the random variable $Y(k)$ be the number of new red free half-edges
when a vertex of degree $k$ is infected. Given the recovery time $\tau$ of
the vertex, $Y(k)\sim\Bin\bigpar{k-1, 1-e^{-\gb_n\tau}}$, and, since
$\tau\sim\Exp(\rho_n)$, the probability $1-e^{-\gb_n\tau}$ has the Beta
distribution $B(1,\rho_n/\gb_n)$. Consequently, $Y(k)$ has the
beta-binomial distribution with parameters $(k-1,1,\rho_n/\beta_n)$.
More generally, if $D$ is a positive integer valued random variable, then
$Y(D)$ denotes a random variable that conditioned on $D=k$ has the
distribution $Y(k)$.
We have the following elementary result, recalling the notation \eqref{pin0}.
\begin{lemma}\label{LY}
For any positive integer valued random variable $D$,
  \begin{align}
\E Y(D) &= \pi_n(\E D-1), \label{eyd}
\\	
\E Y(D)^2 &
=\frac{\pi_n^2\E(D-1)(2D-3)+\pi_n\E(D-1)}{1+\pi_n}. \label{eyd2}
  \end{align}
\end{lemma}

\begin{proof} For each $k\ge1$, we obtain by conditioning on the recovery
  time $\tau$,
%\marginal{Reference?
%http://mathworld.wolfram.com/BetaBinomialDistribution.html?}
  \begin{align}
\E Y(k) &= (k-1) \frac{1}{1+\rho_n/\gb_n}=(k-1)\pi_n,
\\	
\E Y(k)^2 &=  \frac{(k-1)(2k-2+\rho_n/\gb_n)}{(1+\rho_n/\gb_n)(2+\rho_n/\gb_n)}
=\frac{(k-1)(2k-3)\pi_n^2+(k-1)\pi_n}{1+\pi_n}
  \end{align}
and \eqref{eyd}--\eqref{eyd2} follows by conditioning on $D$.
\end{proof}

Let $A$ be a constant, and consider only $m\le M:=\floor{A\ga_n\qww}$. At
pairing event $m$ for $m \le M$, the number of \fhe{s} is at least
$\sum_k k\nk-2A\ga_n\qww\ge \sum_k k\nk\cdot(1-A_1 n\qw\ga_n\qww)$
for some constant $A_1$. Thus, the probability that a susceptible vertex
with $\ell$ half-edges is infected is at most, for $A_2:=2A_1$ and large $n$,
\begin{equation}
  \frac{\ell \nSx{\ell}}{\sum_k k\nk\cdot(1-A_1 n\qw\ga_n\qww)}
\le
 \bigpar{1+A_2 n\qw\ga_n\qww}  \frac{\ell \nSx{\ell}}{\sum_k k\nk}.
\end{equation}
Let $D^+\ge1$ be a random variable with the distribution
\begin{equation}\label{D+}
  \P(D^+\ge j) :=
\min\lrpar{
\bigpar{1+A_2 n\qw\ga_n\qww}\frac{\sum_{k\ge j}k \nSx{k}}{\sum_k k\nk}
,1},
\qquad j\ge2,
\end{equation}
and let $\zeta^+:=Y(D^+)-1$.
(Note that $D^+$ and $\zeta^+$ depend on $n$, although we omit this from the
notation.)
Then $\zeta_m:=\Delta\Zm$, conditioned on what has happened earlier, is
stochastically dominated by $\zeta^+$.
Hence, there exist independent copies $(\zeta^+_m)_1^\infty$
of $\zeta^+$ such that $\zeta_m\le \zeta^+_m$ for all $m\le M$
such that the epidemic has not yet stopped;
furthermore, $(\zeta^+_m)_1^\infty$ are also independent of $Z_0$.
If the epidemic stops at $\mx<M$ (because $Z_{\mx}=0$ so there are no
more pairing events), then we for convenience
extend the definition of $\zeta_m$ and $\Zm$
to all $m\le M$ by defining
$\zeta_m:=\zeta^+_m$  for $m>\mx$,
and still requiring $\zeta_m=\gD\Zm$.
Consequently, $\zeta_m\le\zeta_m^+$ for all $m\le M$ and thus
the (possibly extended) sequence $(\Zm)_0^M$ is
dominated by the random walk $(\Zm^+)_0^M$ with $\Zm^+:=Z_0+\sum_{i=1}^m
\zeta^+_i$.

Next, observe that
\eqref{D+} implies
\begin{equation}\label{ED+>}
  \E D^+-1 = \sum_{j=2}^\infty \P(D^+\ge j)
\ge
\sum_{j=2}^\infty
\frac{\sum_{k\ge j}k \nSx{k}}{\sum_k k\nk}
= \frac{\sum_{k} (k-1)k\nSx{k}}{\sum_k k\nk}
\end{equation}
and also, since $n\ga_n^3\to\infty$,
\begin{equation}\label{ED+<}
%  \E D^+-1 =
\sum_{j=2}^\infty \P(D^+\ge j)
\le
\sum_{j=2}^\infty
\bigpar{1+A_2 n\qw\ga_n\qww}\frac{\sum_{k\ge j}k \nSx{k}}{\sum_k k\nk}
=\bigpar{1+o(\ga_n)} \frac{\sum_{k} (k-1)k\nSx{k}}{\sum_k k\nk}.
\end{equation}
It thus follows that
\begin{equation}\label{ED+}
  \E D^+-1 = \sum_{j=2}^\infty \P(D^+\ge j)
=\bigpar{1+o(\ga_n)} \frac{\sum_{k} (k-1)k\nSx{k}}{\sum_k k\nk}.
\end{equation}
Hence,  using \eqref{eyd}, \eqref{pin0}, \eqref{R0} and \eqref{aR00},
\begin{equation}\label{Ezeta+}
  \begin{split}
  \E \zeta^+
&= \frac{\gb_n}{\gb_n+\rho_n}\E (D^+-1)-1
=
\bigpar{1+o(\ga_n)}\frac{\gb_n}{\gb_n+\rho_n}
\frac{\sum_{k} (k-1)k\nSx{k}}{\sum_k k\nk} -1
\\&
=\bigpar{1+o(\ga_n)}\cR_0-1
=\cR_0-1+o(\ga_n)
%=(\cR_0-1)\bigpar{1+o(1)}.
=\glii\qw\ga_n+o(\ga_n).
  \end{split}
\end{equation}
Note also that \eqref{D+},
using \eqref{e:nSktopk}, \eqref{e:alphanasy}, \eqref{emm},
shows that as \ntoo, $D^+\dto \hDS$,
where $\hDS$ has the size-biased distribution
$\P(\hDS=k)=kp_k/\gl$. Moreover, it follows easily from \ref{d:cubeUI}
and \eqref{D+} that
$D^+$ is uniformly square integrable as \ntoo,  and thus
\begin{equation}\label{ED+2}
  \begin{split}
\E (D^+)^2\to \E(\hDS)^2
=\frac{\sumk k^3p_k}{\sumk kp_k}=\frac{\E\DSl^3}{\E\DSl}
  \end{split}
\end{equation}
and similarly
\begin{equation}\label{ED++}
  \begin{split}
\E D^+\to \E\hDS
%=\frac{\sumk k^2p_k}{\sumk kp_k}
=\frac{\E\DSl^2}{\E\DSl}.
  \end{split}
\end{equation}
By \eqref{Ezeta+} and $\ga_n\to0$ we have $\E\zeta^+\to0$, and thus
$\pi_n(\E D^+-1)\to1$. Hence, by \eqref{ED++} (or directly from \eqref{R0}),
\begin{equation}\label{pin}
  \pi_n =\frac{\gb_n}{\gb_n+\rho_n}\to\frac{\E\DSl}{\E\DSl(\DSl-1)}
=\frac{\gl}{\glii}.
\end{equation}

Since $|\zeta^+|\le D^+$, it follows that also
$\zeta^+$ is uniformly square integrable as \ntoo. Furthermore,
$\E Y(D^+) = 1+\E \zeta^+=1+o(1)$ and thus, using \eqref{eyd2} and
\eqref{ED+2}--\eqref{pin},
\begin{equation}\label{vz+}
  \begin{split}
\Var\zeta^+
&=
\Var(Y(D^+))=
\E (Y(D^+)^2)-\bigpar{1+o(1)}^2
\\&
= \frac{\pi_n^2}{1+\pi_n}\E(D^+-1)(2D^+-3) + \frac{\pi_n}{1+\pi_n}\E(D^+-1)
-1+o(1)
\\&
\to
\frac{\gl^2}{\glii(\glii+\gl)}\frac{\E\DSl(\DSl-1)(2\DSl-3)}{\E\DSl}
+
\frac{\gl}{\glii+\gl}\frac{\E\DSl(\DSl-1)}{\E\DSl}
-1
\\&
=\frac{\gl(2\gliii+\glii)}{\glii(\glii+\gl)}
+
\frac{\glii}{\glii+\gl}
-1
\\&
=\frac{2\gl\gliii}{\glii(\glii+\gl)}
=:\gss.
  \end{split}
\end{equation}

Now consider
$\ga_n(\Zx{M}^+ -Z_0-M\E\zeta^+)
=\ga_n\sum_{i=1}^{\floor{A\ga_n\qww}}(\zeta^+_i-\E\zeta^+)$.
The summands are
i.i.d.\ with mean 0, and the uniform square integrability of $\zeta^+$
implies that the Lindeberg condition holds; thus the central limit theorem
\cite[Theorem 5.12]{kallenberg} applies and yields, using \eqref{vz+},
$\ga_n(\Zx{M}^+ -Z_0-M\E\zeta^+)\dto N(0,A\gss)$ as \ntoo.
Moreover, normal convergence of the endpoint of a random walk implies
Donsker-type convergence of the entire random walk to a Brownian motion, see
\cite[Theorem 14.20]{kallenberg}; hence,
\begin{equation}\label{z+lim1}
  \ga_n\lrpar{\Zx{t\ga_n\qww}^+ -Z_0-t\ga_n\qww\E\zeta^+}\to \gs B_t,
\end{equation}
where $B_t$ is a standard Brownian motion
and we have defined $\Zx{t}^+$ also
for non-integer $t$ by $\Zx{t}^+:=\Zx{\floor{t}}^+$.
(We define $\Zx{t}$ and $\Zx{t}^-$ below in the same way.)
Here the convergence is in distribution in the \Skorohod{} space $D[0,A]$,
but we may by the \Skorohod{} coupling theorem
\cite[Theorem 4.30]{kallenberg}
assume that the processes for different $n$ are coupled such that \as{}
\eqref{z+lim1} holds uniformly on $[0,A]$.

Moreover, $\ga_n\qw\E\zeta^+\to\glii\qw$ by \eqref{Ezeta+}, and
thus \eqref{z+lim1} implies
\begin{equation}\label{z+lim}
  \ga_n\lrpar{\Zx{t\ga_n\qww}^+ -Z_0}\to \gs B_t + \glii\qw t .
\end{equation}

\begin{proof}[Proof of Theorem~\ref{t:takeoff}\ref{takeoff1}.]
In this case, $\ga_n\XI0\to0$, and, since $0\le Z_0\le\XI0$,
it follows from \eqref{z+lim} that
$  \ga_n\Zx{t\ga_n\qww}^+ \to \gs B_t + \glii\qw t$, where (as said above) we
may assume that the convergence holds uniformly on $[0,A]$ a.s.
For any fixed $\deps>0$,  the \rhs{} is \as{} negative for some
$t\in[0,\deps]$, and thus \whp{}
$\ga_n\Zx{t\ga_n\qww}^+ <0$ for some $t\in[0,\deps]$.
Since $Z_m\le\Zm$
it follows that \whp{}
$\mx\le\deps\ga_n\qww$,
\ie{}
the epidemic stops with
$Z_m=0$
after at most
$\deps\ga_n\qww$ infections.
Hence, \whp{}
\begin{equation}\label{t10}
  \cZ\le \mx\le \deps\ga_n\qww =o\bigpar{n\ga_n}.
\end{equation}
Since $\gd$ is arbitrary, this moreover shows
$\fs=\op\bigpar{\gan\qww}$.
%at most $\deps\ga_n\qww$ susceptible vertices are ever infected.
\end{proof}

\begin{proof}[Proof of Theorem~\ref{t:takeoff}\ref{takeoff3}
in the multigraph case.]
We combine the upper bound $\Zm^+$ above with a matching lower bound.
Let $x_1,x_2,\dots$ be an i.i.d.\ sequence of random half-edges,
constructed before we run the epidemic by drawing with replacement from the
set of all half-edges.
Then, at the $m$:th pairing event, when we are to pair an infective
red half-edge $y_m$, if $x_m$ still is free and $x_m\neq y_m$, we pair $y_m$
with $x_m$;  otherwise we
resample and pair $y_m$ with a uniformly chosen \fhe{} $\neq y_m$.
Furthermore,
we let
$\zeta_m^-:=-1$
if  $x_m$
is initially infective,
and if $x_m$ belongs to an initially susceptible vertex of degree $k$,
we let $\zeta_m^-$ be a copy of $Y(k)-1$ (independent of the history);
if $x_m$ still is susceptible
at the $m$:th pairing event (and thus free, so we pair with $x_m$),
we may assume that $\zeta_m^-:=\zeta_m$, the
number of new red \fhe{s} minus 1.
Note that $(\zeta_m^-)_{m\ge1}$ is an \iid{} sequence of random variables
with the distribution $Y(D^-)-1$, where $D^-$ has the distribution obtained
by taking $A_2=0$ in \eqref{D+};
furthermore, $(\zeta^-_m)_1^\infty$ are independent of $Z_0$.
Let $\Zm^-:=Z_0+\sum_{i=1}^m\zeta_i^-$.
Note that \eqref{ED+>}--\eqref{vz+} hold for $D^-$ and $\zeta_m^-$ too (with
some simplifications), and thus, in analogy with \eqref{z+lim},
\begin{equation}\label{z-lim}
  \ga_n\lrpar{\Zx{t\ga_n\qww}^- -Z_0}\to \gs B^-_t + \glii\qw t ,
\end{equation}
for some Brownian motion $B_t^-$.
We next verify that we can take the same Brownian motion in \eqref{z+lim}
and \eqref{z-lim}.

Let $\zeta_m':=\zeta_m^--\zeta_m$.
Thus $\zeta_m'=0$ if $x_m$ is susceptible at time $T_m$. If $x_m$ was
initially susceptible, with degree $k$, but has been infected, then
$\zeta_m'\le \zeta_m^-+2\le k$.
If $x_m$ was initially infected, then $\zeta_m^-=-1$ and thus
$\zeta_m'\le \zeta_m^-+2\le 1$.

Consider as above only $m\le M:=\floor{A\ga_n\qww}$, for some (large)
constant $A>0$. For $m>\mx$, when the epidemic has stopped, we have
defined $\zeta_m=\zeta_m^+$. Since $\zeta_m^\pm\eqd Y(D^\pm)-1$ and $D^-$ is
stochastically dominated by $D^+$, we may in this case assume that
$\zeta_m=\zeta_m^+\ge\zeta_m^-$, and thus $\zeta_m'\le0$.

For $m\le M$, the number of
initially susceptible half-edges that have been infected
is at most,
using \eqref{RdSmax} and \eqref{e:alphanasy},
$ m \dSmax = O\bigpar{\ga_n\qww\dSmax}
= o\bigpar{\ga_n\qww n^{1/3}}=o(n)$.
Hence the number of \fhe{s}  at $T_m$ is
at least $\sum_k k\nSk-m \dSmax=\gl n-o(n)\ge \cc n \ccdef\ccm$
for $\ccm:=\gl/2$ if $n$ is large enough.
It follows that the probability that a given initially susceptible vertex of
degree $k$ has been infected before $T_m$ is at most
$mk/(\ccm n)$, and the probability that one of its half-edges is chosen as
$x_m$ is at most $k/(\ccm n)$ for every $m\le M$.
Similarly, the probability that $x_m$ is initially infective is at most
$\XI0/(\ccm n)$.

Hence it follows from the comments above, using \eqref{k3} and
the assumption $\ga_n\XI0=O(1)$ in \ref{takeoff3}, that
$(\zeta_m')_+:=\max(\zeta_m',0)$ has expectation
\begin{equation}\label{Ezeta'}
  \E(\zeta_m')_+
\le
\sumk n_k \frac{k}{\ccm n}\cdot\frac{mk}{\ccm n}\cdot k
+\frac{\XI0}{\ccm n}
=O\Bigpar{\frac{m}{n}}+O\Bigparfrac{\ga_n\qw}{n}
=O\Bigpar{\frac{1}{\ga_n^2n}}
=o(\ga_n).
\end{equation}
Let $\bZ':=\sum_{1}^M(\zeta'_m)_+$. Then, by \eqref{Ezeta'},
\begin{equation}
  \label{EZ'}
\E \bZ' = o(M\ga_n)=o\bigpar{\ga_n\qw}.
\end{equation}
Furthermore, for $m\le M$,
\begin{equation}
  \label{zzz}
\Zm^- = \Zm+\sum_{j=1}^m \zeta'_j \le \Zm+\bZ'\le \Zm^++\bZ'.
\end{equation}

Since \eqref{z+lim} and \eqref{z-lim} hold (in distribution),
the sequence $\bigpar{\ga_n\bigpar{\Zx{t\ga_n\qww}^- -Z_0}
,\ga_n\bigpar{\Zx{t\ga_n\qww}^+ -Z_0}}$, $n\ge1$, is tight in $D[0,A]\times
D[0,A]$. Moreover, every subsequential limit in distribution must be of the
form
$\bigpar{\gs B^-_t + \glii\qw t ,\gs B^+_t + \glii\qw t }$ for some Brownian
motions $B^-_t$ and $B^+_t$.
Since $\ga_n\bZ'\pto0$ by \eqref{EZ'}, it then follows from  \eqref{zzz} that
for any fixed $t\in[0,A]$, $B_t^-\le B_t^+$ a.s. Since $B_t^-$ and $B_t^+$
have the same distribution, this implies $B_t^-=B_t^+$ \as{} for every fixed
$t$, and thus by continuity \as{} for all $t\in[0,A]$.

Since all subsequential limits thus are the same, this shows that
\eqref{z+lim} and \eqref{z-lim} hold jointly (in distribution) with $B_t^-=B_t$.
%and by the \Skorohod{} representation theorem we may assume that they hold
%uniformly in $t\in[0,A]$ a.s.
Finally, by \eqref{zzz} and \eqref{EZ'}, this implies
\begin{equation}\label{zlim}
  \ga_n\lrpar{\Zx{t\ga_n\qww} -Z_0}\dto \gs B_t + \glii\qw t ,
\qquad \text{in $D[0,A]$}.
\end{equation}
Since the infimum is a continuous functional on $D[0,A]$, it follows that
\begin{equation}\label{zlim1}
  \ga_n\Bigpar{\inf_{t\le A}\Zx{t\ga_n\qww} -Z_0}
\dto\inf_{t\le A}\bigpar{ \gs B_t + \glii\qw t}
.
\end{equation}
For convenience, denote the left- and \rhs{s} of \eqref{zlim1} by $Y_n$ and
$Y$.
Since the random variable $Y$ has a continuous
distribution, \eqref{zlim1} implies that, uniformly in $x\in\bbR$,
\begin{equation}\label{zlim2}
  \P\bigpar{Y_n \le x}=\P( Y\le x)+o(1).
\end{equation}
The Brownian motion $B_t$ in \eqref{zlim}--\eqref{zlim1} is arbitrary, so we
may and shall assume that $B_t$ is independent of everything else.

We have defined $\mx:=\min\set{m\ge0:\Zm=0}$ and $M:=\floor{A\ga_n\qww}$,
and thus
\begin{equation}\label{mmm}
%\mx\le M
%\iff
%\Zx{m}=0 \text{ for some }m\le M
%\iff
%\inf_{t\le A}\Zx{t\ga_n\qww}\le0
%\iff Y_n\le -\ga_n Z_0.
  \P\bigpar{\mx\le M}
=\P\Bigpar{\inf_{t\le A}\Zx{t\ga_n\qww}\le 0}
=\P\bigpar{Y_n\le-\ga_n Z_0}.
\end{equation}

Recall that $\zeta^+_m$ and $\zeta^-_m$ above are independent
of $Z_0$. Hence, if we fix two real numbers $a$ and $b$, and
condition on the event $\cenab:=\set{a\le \ga_n Z_0 <b}$,
then for every subsequence such that $\liminf_\ntoo\P(\cenab)>0$,
the arguments above leading to \eqref{z+lim}, \eqref{z-lim} and
\eqref{zlim}--\eqref{zlim2} still hold.
(We need $\liminf_\ntoo\P(\cenab)>0$ in order to get a conditional version
of \eqref{EZ'}.) Consequently,
$%\begin{equation}
  \P\bigpar{Y_n \le x\mid \cenab}=\P( Y\le x)+o(1)
$, %\end{equation}
and thus, recalling that $B_t$ is independent of $Z_0$,
\begin{equation}\label{puh}
  \P\bigpar{Y_n \le x\tand \cenab}=\P( Y\le x)\P(\cenab)+o(1)
=\P( Y\le x\tand\cenab)+o(1).
\end{equation}
On the other hand, \eqref{puh} holds trivially if $\P(\cenab)\to0$.
Every subsequence has a subsubsequence such that either
$\liminf_\ntoo\P(\cenab)>0$ or $\P(\cenab)\to0$, and in any case \eqref{puh}
holds along the subsubsequence; it follows that \eqref{puh} holds for the
full sequence.

In particular, for any $a$ and $b$,
\begin{equation}\label{kri}
  \begin{split}
\P\bigpar{Y_n\le -\gan Z_0\tand\cenab}	
&\le
\P\bigpar{Y_n\le -a\tand\cenab}	
=
\P\bigpar{Y\le -a\tand\cenab}+o(1)	
\\&
\le
\P\bigpar{Y\le -\gan Z_0+b-a\tand\cenab}+o(1).	
  \end{split}
\end{equation}
By assumption, $\gan\XI0$ is bounded, say $\gan\XI0\le C$ for some constant
$C$;
thus $0\le \gan Z_0\le\gan\XI0\le C$.
Let $\gd>0$ and
divide the interval $[0,C]$ into a finite number of subintervals $[a_j,b_j]$
with lengths $b_j-a_j<\gd$. By summing \eqref{kri} for these intervals, we
obtain
\begin{equation}\label{lutfisk<}
  \begin{split}
\P\bigpar{Y_n\le -\gan Z_0}	
\le
\P\bigpar{Y\le -\gan Z_0+\gd}+o(1).	
  \end{split}
\end{equation}
Since $\gd>0$ is arbitrary, this implies
\begin{equation}\label{anna<}
  \begin{split}
\P\bigpar{Y_n\le -\gan Z_0}	
\le
\P\bigpar{Y\le -\gan Z_0}+o(1).	
  \end{split}
\end{equation}
Similarly, we obtain
$\P\bigpar{Y_n\le -\gan Z_0}	
\ge
\P\bigpar{Y\le -\gan Z_0-\gd}+o(1)$
and
$\P\bigpar{Y_n\le -\gan Z_0}	
\ge
\P\bigpar{Y\le -\gan Z_0}+o(1)$.
Consequently,
\begin{equation}
\P\bigpar{Y_n\le -\gan Z_0}	
=
\P\bigpar{Y\le -\gan Z_0}+o(1).
\end{equation}
In other words, using \eqref{mmm} and recalling the meaning of $Y$ from
\eqref{zlim1},
\begin{equation}\label{mx}
  \begin{split}
  \P\bigpar{\mx\le M}
%&=\P\Bigpar{\inf_{t\le A}\Zx{t\ga_n\qww}\le 0}
%=\P\bigpar{Y_n\le-\ga_n Z_0}
%\\&
= \P \lrpar{\inf_{t\le A}\bigpar{\gs B_t + \glii\qw t }\le-\ga_n Z_0}+o(1).	
  \end{split}
\end{equation}

If $\mx\le M=\floor{A\gan\qww}$, then,  similarly as  %\cf{}
\eqref{t10} in the proof of case \ref{takeoff1},
\begin{equation}\label{t13}
  \cZ\le \mx\le A\ga_n\qww =o\bigpar{n\ga_n}
\end{equation}
so we are in
case \ref{tsizei:small} in \refT{t:criticalSIR}\ref{tsizei}.

If $\mx>M$, consider again $\mxx$ defined by \eqref{mxx} (but taking minimum
over $m \ge M$),
for a sufficiently
small $\gd>0$. Note that, as in the proof of \ref{takeoff2},
if $\Zx{\mxx}>0$, then \eqref{paddington} holds and \whp{} $\iEnd>\eps\ga_n$
for some small $\eps>0$, and thus \whp{}
\ref{tsizei:large} in \refT{t:criticalSIR}\ref{tsizei} holds.
In other words, for some small $\eps>0$, if
$\mx\le M$, then $\cZ<\eps n \gan$, and if
$\mx>M$ and $\Zx{\mxx}>0$, then $\cZ>\eps n\gan$ \whp.

We next show that the probability that neither of these happens is small. We
condition on $\Zx M$ and argue as in the proof of case \ref{takeoff2}, using
\refL{l:RWdrift} on $\Zx{(M+m)\wedge\mxx}-\Zx M$, and find
\begin{equation}
  \P\bigpar{\mx>M\text{ and } \Zx{\mxx}=0\mid\Zx M}
\le \frac{8\cczmii}{\cczmi\ga_n\Zx M}
.
\end{equation}
Hence, using also \eqref{zlim},
\begin{equation}\label{tt21}
  \begin{split}
\P\bigpar{\mx>M\text{ and } \Zx{\mxx}=0}
&\le \P(\gan\Zx M < \tfrac12\glii\qw A)
+ O(1/A)
\\&
\le \P\bigpar{\gs B_A+\glii\qw A < \tfrac12\glii\qw A}+o(1)
+ O(1/A)
\\&
=O(1/A)+o(1)
.
  \end{split}
\end{equation}
Using \eqref{mx} and the comments above, it follows
that, if $\eps>0$ is small enough, then
\begin{equation}
  \begin{split}
  \P\xpar{\cZ<\eps\gan n}
&=\P\bigpar{\mx\le M}+O(1/A)+o(1)
  \\&
= \P \lrpar{\inf_{t\le A}\bigpar{\gs B_t + \glii\qw t }\le-\ga_n Z_0}
+O(1/A)+o(1).	
  \end{split}
\end{equation}
This holds for every fixed $A>0$, and we can then let $A\to\infty$ and
conclude that
\begin{equation}
  \begin{split}
  \P\xpar{\cZ<\eps\gan n}
%&= \P \lrpar{\inf_{0\le t<\infty}\bigpar{\ga_n Z_0+\gs B_t + \glii\qw t }\le0}
%+o(1)
%\\&
= \P \lrpar{\inf_{0\le t<\infty}\bigpar{\gs B_t + \glii\qw t }\le -\ga_n Z_0}
+o(1).
  \end{split}
\end{equation}
It is well-known that
$-\inf_{t\ge0}\bigpar{\gs B_t + \glii\qw t }$ has an exponential distribution
with parameter $2\glii\qw/\gss$, see \eg{} \cite[Exercise II.(3.12)]{RevuzYor}.
Consequently, since $Z_0$ and $(B_t)$ are independent,
\begin{equation}\label{psmall}
  \begin{split}
  \P\xpar{\cZ<\eps\gan n}
&= \E \exp\bigpar{-2\glii\qw\gs\qww\gan Z_0}
+o(1).
  \end{split}
\end{equation}
Since we assume that $\gan\XI0$ is bounded above and below,
$Z_0\le\XI0$ and \refL{l:alphaZ0}\ref{l:alphaZ02}
imply that the expectation in \eqref{psmall} stays away
from 0 and 1 as \ntoo. Moreover, if $\dImax=o(\XI0)$, then
\refL{l:alphaZ0}\ref{l:alphaZ01} and \eqref{psmall} yield,
using \eqref{pin},
\begin{equation}\label{psmall2}
  \begin{split}
  \P\xpar{\cZ<\eps\gan n}
&= \exp\bigpar{-2\glii\qw\gs\qww\gan\pi_n\XI0}+o(1)
=  \exp\bigpar{-2\gl\glii\qww\gs\qww\gan\XI0}+o(1),
  \end{split}
\end{equation}
which yields \eqref{takeoff} by
 the definition of $\gss$ in \eqref{vz+}.

Finally, \eqref{t13} and the argument above, in particular \eqref{tt21},
shows that
\begin{equation}
  \P\bigpar{\text{the epidemic is small but }\fs>A\gan\qww}=O(1/A)+o(1),
\end{equation}
which implies the final claim.
\end{proof}

\begin{proof}[Proof of Theorem~\ref{t:takeoff}\ref{takeoff3}
in the simple graph case.]
As said in \refS{s:notationresults}, this result for the random simple graph
$\Gr$ does not follow immediately from the multigraph case (as the other
results in this paper do).
We use here instead the argument for the corresponding result in
\cite[Section 6]{JansonLuczakWindridgeSIR}, with minor modifications as
follows. We continue to work with the random multigraph $\GrCM$.
Also, we now allow initially recovered vertices, since our trick
in \refS{s:simplify}
to eliminate them does not work for the simple graph case.

Fix a sequence $\eps_n\to0$ such that \refT{t:criticalSIR}\ref{tsizei}
holds,
and let $\cL$ be the event that there are less than $\eps_n\qq\nS\gan$ pairing
events; note that if $\cL$ occurs, then $\fs<\eps_n\qq\nS\gan$,
while if $\cL$ does not occur, \whp{} $\fs >\eps_n\nS\gan$ by a simple
argument (using \eg{} Chebyshev's inequality); hence $\cL$ says \whp{} that
the epidemic is small.

Furthermore, let $W$ be the number of loops and pairs of parallel edges in
$\GrCM$; thus $\GrCM$ is simple if and only if $W=0$, and we are interested
in the conditional probability $\P(\cL\mid W=0)$.
By \cite{simpleII} (at least if we consider suitable subsequences), $W\dto
\hW$ for some random variable $\hW$, with convergence of all moments.

We write $W=W_1+W_2$, where $W_2$ is the number of loops and pairs of
parallel edges that include either an initially infective vertex
(as in \cite{JansonLuczakWindridgeSIR}), or a vertex with degree at least
$\dxx:=1/\ga_n$.
Then, by the assumptions
\begin{equation}
  \E W_2 = O\lrpar{\biggpar{\sumk k^2 \nIk + \sum_{k\ge\dxx}k^2(\nSk+\nRk) }
\biggpar{\frac{1}{n}+\frac{\sumk k^2\nk}{n^2}}}
=o(1)
\end{equation}
and thus it suffices to consider $W_1$.
Note also that if we fix a vertex $v$ that is not initially infected and has
degree less than $\dxx$, then
the probability that the infection will reach $v$ within less than
$\eps_n\qq \nS\gan$ pairing events is
$O\bigpar{\dxx\eps_n\qq n \gan/n}=o(1)$,
so \whp{} $v$ is not infected before it is
determined whether $\cL$ occurs or not.

The rest of the proof is exactly as in
\cite{JansonLuczakWindridgeSIR},
to which we refer for details.
\end{proof}

\begin{remark}\label{Rtakeoff}
The formula \eqref{takeoff} for the asymptotic
probability that the epidemic is small
holds only under the assumption $\dImax=o(\XI0)$, \ie, that among the initially
infective vertices, no vertex has a significant fraction of all their
half-edges. Even if this assumption does not hold, the asymptotic
probability
can be found from \eqref{psmall}, since as in the proof of \refL{l:alphaZ0},
$Z_0=\sum_i Z_{0,i}$ where the $Z_{0,i}$ are independent and, using the
notation in \refL{LY},
$Z_{0,i}\eqd Y(\dIi+1)$, where $\dIi$ is the degree of the $i$-th initially
infective vertex. Hence, letting $\kk$ denote the fraction in
%the exponent in
\eqref{takeoff}, so $\kk\sim 2\glii\qw\gs\qww\pi_n$,
the probability is
\begin{equation}\label{lyman}
  \begin{split}
\prod_i\E \exp\bigpar{-\kk\pi_n\qw\gan Y(\dIi+1)}+o(1)
=
\prod_{k}\bigpar{\E \exp\bigpar{-\kk\pi_n\qw\gan Y(k+1)}}^{\nIk}+o(1).
  \end{split}
\end{equation}
A calculation,
%which we omit,
see \refApp{Adeluxe},
shows that if we define
\begin{equation}\label{psin}
\psi_n(k):=  \log\intoi
\exp
\Bigpar{k\gan\kk\pi_n\qw\bigpar{x^{\gb_n/\rho_n}-\tfrac{\rho_n}{\gb_n+\rho_n}}}
\,dx,
\end{equation}
interpreted as 0 when $\rho_n=0$,
then this probability is
\begin{equation}\label{takeoffDeluxe}
  \exp\Bigpar{-\kk\gan\XI0+\sum_k \nIk \psi_n(k) }+o(1),
\end{equation}
thus generalizing \eqref{takeoff}.
It is easily seen that $\psi_n(k)=O\bigpar{k^2\gan^2}$ and thus
$\sum_k \nIk \psi_n(k)=O\bigpar{\gan\dImax\sum_k \nIk k  \gan}
=O\bigpar{\gan\dImax}$ under our assumption $\gan\XI0=O(1)$, which
explains why the extra term in \eqref{takeoffDeluxe}
disappears in \eqref{takeoff}.

Note also that $\psi_n(k)\ge0$ by Jensen's inequality; thus an extremely uneven
distribution of the degrees of the initially infective vertices will
increase the probability of a small outbreak.
\end{remark}

\appendix

\section{The Sellke construction for network epidemics}

\label{app:sellke}

House et al.~\cite{House:2012} introduced a method to simulate the final size
of a network epidemic that drew on the constructions of
Sellke~\cite{Sellke:1983} and Ludwig~\cite{Ludwig:1975} (the latter in its most
general sense as described by Pellis et al.~\cite{Pellis:2008}). This approach
allows us to make a realisation by drawing three sets of random numbers, and
then we can consider multiple initial conditions and the final sizes they
produce without drawing further random numbers.

Let individuals be labelled $i,j,\ldots \in \{ 1, \ldots, n\}$. Let them be
connected by a general network with adjacency matrix with elements $G_{ij}$,
and let $Z_i$ be an indicator variable taking the value $1$ if individual $i$
is eventually infected during the epidemic and $0$ otherwise. Our algorithm
then proceeds as follows.

First, for each individual $i$ pick an infectious period $T_i \sim
\mathrm{Exp}(\rho)$. Secondly, for each individual $i$ pick a threshold
$Q_i\sim\mathrm{Exp}(1)$.  This represents the individual's resistance to
infection.  Thirdly, a random permulation $P$ of the integers $\{1,\ldots,n\}$
is chosen. The first $m$ elements of this permutation are then taken to be the
indices of the initially infectious individuals, i.e.\ we initialise
$Z_i\leftarrow 1$ if $i\in P$ and $Z_i\leftarrow 0$ otherwise. Then to arrive
at the correct final values we iteratively search for each individual $i$ that
has $Z_i=0$ and set $Z_i\leftarrow 1$ if
\begin{equation}
	Q_i < \beta \sum_j G_{ij} T_j Z_j \; . \label{sel_net}
\end{equation}
This procedure is continued until no changes occur on a given iteration, giving
the final size for that value of $m$. We can
then increase $m$ and continue the iterative procedure, allowing simulations
that are both computationally efficient, and for which the total final size
$\mathcal{Z}$ is monotone non-decreasing in the initial number infected $m$.

\section{Critical behaviour of $\Gr(n, (d_i)_{i=1}^n)$}
\label{app:cpts}
In this appendix, we show how the methods in this paper  yield an
improvement of
\cite[Theorem 2.4]{JansonLuczakGiantCpt}
on the size of the largest component in $\Gr(n, (d_i)_{i=1}^n)$ near
criticality,
replacing the moment condition
used in \cite{JansonLuczakGiantCpt}
\begin{equation}\label{e:4+eta}
\sumk k^{4+\eta} \nk = O(n)
\end{equation}
for some $\eta > 0$,
by the uniform summability of $\sumk k^{3} \nk$
in assumption \eqref{d:uicubenk} below, \cf{} \ref{d:cubeUI}.
This is more or less best possible,
since if
the limiting vertex degree distribution does not have a finite third moment,
then (at least in typical cases), the size of the largest component is
$\op(n\ga_n)$,
see van der Hofstad, Janson and Luczak (in preparation).

\begin{theorem}\label{t:largecpts}
Suppose that the degree sequences $(d_i)_{i=1}^n$ satisfy the following conditions.
\begin{enumerate}[\upshape(i)]
\item There is a probability distribution $(p_k)_{k = 0}^\infty$ such that $\nk/n \to p_k$ for each $k \ge 0$.
%\item The mean $\lambda' \= \sumk k p'_k$ is finite.
\item\label{d:uicubenk} For all $\eps > 0$, there exists $M$ such that, for all $n$, $\sum_{k=M}^\infty k^3 \nk < \eps n$.
\item $p_1 > 0$.
\item $\sumk k(k-2)p_k = 0$.
\item $\alpha_n \= \sumk k(k-2)\nk/n$ satisfies $n^{1/3} \alpha_n \to \infty$.
\end{enumerate}

Define the positive constants $\lambda \= \sumk k p_k$ and $\gamma \= \sumk
k(k-1)(k-2)p_k$.
%Then $\lambda, \gamma \in (0,\infty)$.

Let $\cC_1$ and $\cC_2$ denote the largest and second largest components of $\Gr(n, (d_i)_{i = 1}^n)$.
%the number of vertices in $\cC_1$ is
Then the number of vertices in $\cC_1$ is
\begin{equation}
v(\cC_1) = \frac{2\lambda}{\gamma}n\alpha_n + \op(n\alpha_n),
\end{equation}
the number of vertices in $\cC_1$ with degree $k \ge 0$ is
\begin{equation}
v_k(\cC_1) = \frac{2k p_k}{\gamma}n\alpha_n + \op(n\alpha_n),
\end{equation}
and the number of edges in $\cC_1$ is
\begin{equation}
e(\cC_1) = \frac{2\lambda}{\gamma}n\alpha_n + \op(n\alpha_n).
\end{equation}

In contrast, $\cC_2$ has only $v(\cC_2) = \op(n\alpha_n)$ vertices and
$e(\cC_2) = \op(n\alpha_n)$ edges.
\end{theorem}

\begin{remark}
If $\sum_k k (k-2) p_k = 0$ then $p_1 > 0$ is equivalent to $p_0+p_1 +p_2 <
1$ (condition \ref{d:p1orrhopos}).
As a consequence, $\gam>0$.
\end{remark}

\begin{proof}[Sketch of proof]
We will explain how to modify the argument of \cite{JansonLuczakGiantCpt}.
The latter is similar in spirit (and actually inspired) our proof of
\refT{t:criticalSIR} in \refS{s:proof}.
Indeed, the algorithm used in \cite{JansonLuczakGiantCpt} to construct the
multigraph $\GrCM(n, (d_i)_{i=1}^n)$ and explore its components is closely related
to the time-changed epidemic with zero recovery rate, i.e.\ $\rho_n = 0$.  In more detail,
we begin with all vertices susceptible (or \emph{sleeping} in the terminology of \cite{JansonLuczakGiantCpt}).  We choose a single vertex at random and declare it infective (or \emph{active} in \cite{JansonLuczakGiantCpt}).  The infection eventually spreads through the whole connected component containing the chosen vertex because $\rho_n = 0$.
The connected component thus comprises the susceptible (sleeping) vertices that were infected (activated)
during this time, together with the initially infective (active) vertex.  The procedure can be repeated until all the connected components have been explored.

In particular, each susceptible (sleeping) vertex of degree $k \ge 0$ is still infected (activated)
at rate $k$ in the algorithm of \cite{JansonLuczakGiantCpt}; in addition one
vertex is activated at the start of each new component.
The number of vertices that would be sleeping if we ignore the latter type
of activations
(denoted $\tilde V_k(t)$ in \cite{JansonLuczakGiantCpt})
thus evolves as the Markov death chain $\Svkm{t}$ considered in the proof of
\refT{t:conc}
but with $\Svkm{0} = \nk$.  One can prove concentration of measure for
$\Svkm{t}$, $\sumk \Svkm{t}$ and $\sumk k \Svkm{t}$, as in \refT{t:conc}.
The analogous result in \cite{JansonLuczakGiantCpt}
is Lemma 6.3 and its proof involves a non-trivial use of assumption
\eqref{e:4+eta}.
So we must replace \cite[Lemma 6.3]{JansonLuczakGiantCpt} with \refT{t:conc}.  Note that the former result
achieves an $\Op(n\qq\gan\qq+n\alpha_n^3)$
bound on the error, versus the $\op(n\alpha_n^2)$ bound of \refT{t:conc}.  However,
it can be checked that the $\op(n\alpha_n^2)$ bound is sufficient for the rest of the proof.

There is only one other place where \cite{JansonLuczakGiantCpt} uses assumption \eqref{e:4+eta}
non-trivially; that is to control the Taylor expansion remainder term in the analogue of \eqref{HITaylor}
(immediately preceding equation (6.7) on page 212).  But we may obtain an adequate bound with the argument leading
to \eqref{e:dddHIntBD}.

All of the other modifications are trivial.
\end{proof}

\section{Proof of \eqref{takeoffDeluxe}}\label{Adeluxe}

Let $c_n:=\kk\pi_n\qw$ and note that $c_n=O(1)$ by \eqref{pin0} and
\eqref{rhobeta}.
If $\rho_n>0$, then by the definition of $Y_n(k)$ before \refL{LY},
and the substitution $x=e^{-\rho_n\tau}$,
\begin{equation}\label{nobel}
  \begin{split}
	\E e^{-c_n\ga_n Y(k+1)}
&=\intoo\Bigpar{1+\bigpar{1-e^{-\gb_n\tau}}\bigpar{e^{-c_n\gan}-1}}^k
 \rho_n e^{-\rho_n\tau}\,d\tau
\\&
=\intoi \Bigpar{1-\bigpar{1-x^{\gb_n/\rho_n}}\bigpar{1-e^{-c_n\gan}}}^k \,dx
\\&
=\intoi \Bigpar{1-\bigpar{1-x^{\gb_n/\rho_n}}\bigpar{c_n\gan(1+O(\gan))}}^k \,dx
\\&
=\intoi \exp\Bigpar{-k\bigpar{1-x^{\gb_n/\rho_n}}\bigpar{c_n\gan+O(\gan^2)}} \,dx
\\&
=e^{O(k\gan^2)} \intoi
\exp\Bigpar{c_n\gan k\bigpar{x^{\gb_n/\rho_n}-1}} \,dx
\\&
=e^{-c_n\pi_n\gan k+O(k\gan^2)} \intoi
\exp\Bigpar{c_n\gan k\bigpar{x^{\gb_n/\rho_n}-\tfrac{\rho_n}{\gb_n+\rho_n}}}
\,dx
\\&
=e^{-\kk \gan k+\psi_n(k)+O(k\gan^2)}
.
%\intoi \exp\Bigpar{\kk\pi_n\qw\gan k
%    \bigpar{x^{\gb_n/\rho_n}-\tfrac{\rho_n}{\gb_n+\rho_n}}}  \,dx
  \end{split}
\end{equation}
If $\rho_n=0$, then $\pi_n=1$ and $Y(k+1)=k$, and \eqref{nobel} is trivial.

We obtain \eqref{takeoffDeluxe} by using \eqref{nobel} in \eqref{lyman}.

%\bibliography{critical-sir-cm-jlw}
%\bibliographystyle{plain}

\begin{figure}[h!]
\centering
\includegraphics[width = 0.45\textwidth]{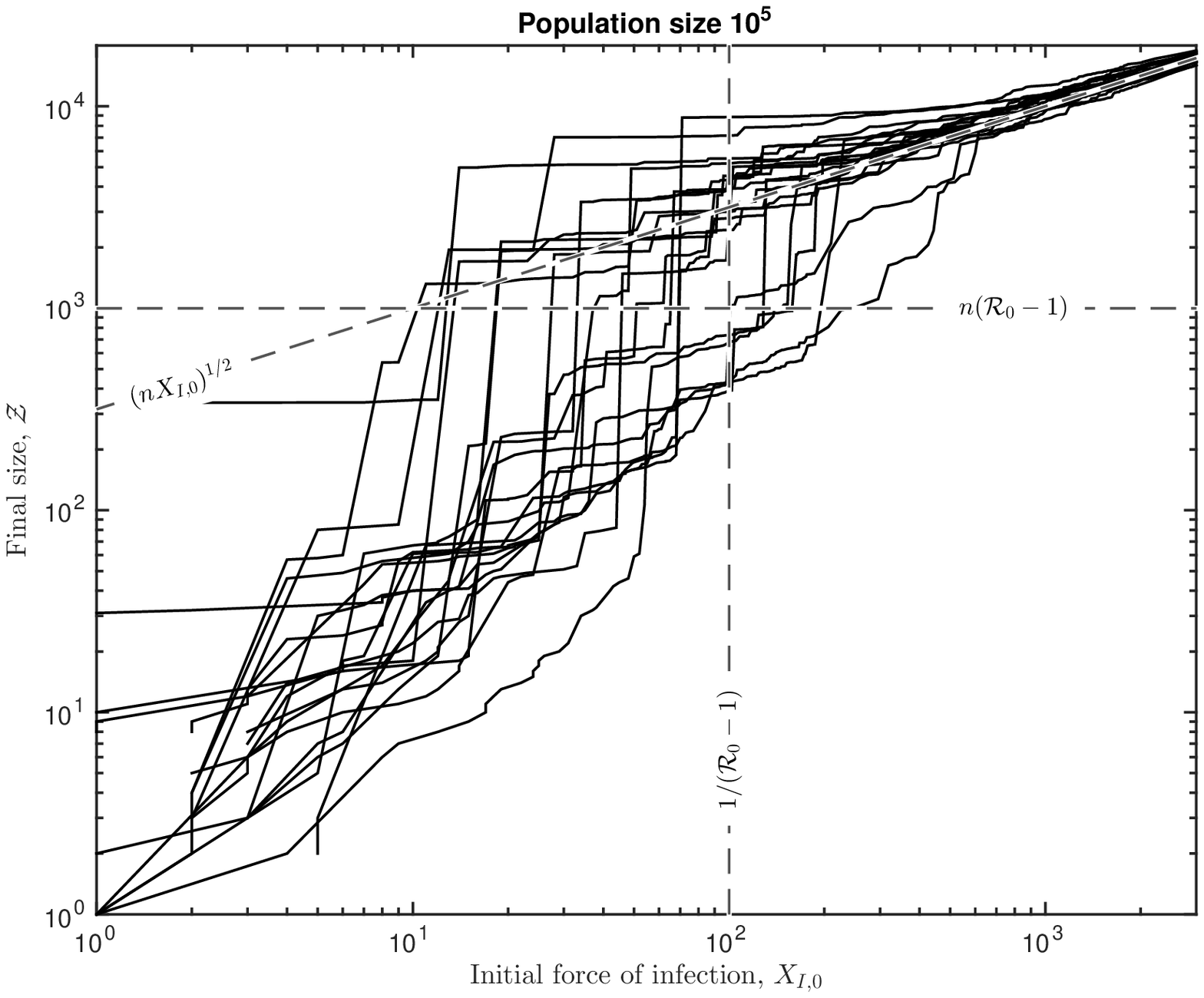}\quad
\includegraphics[width = 0.45\textwidth]{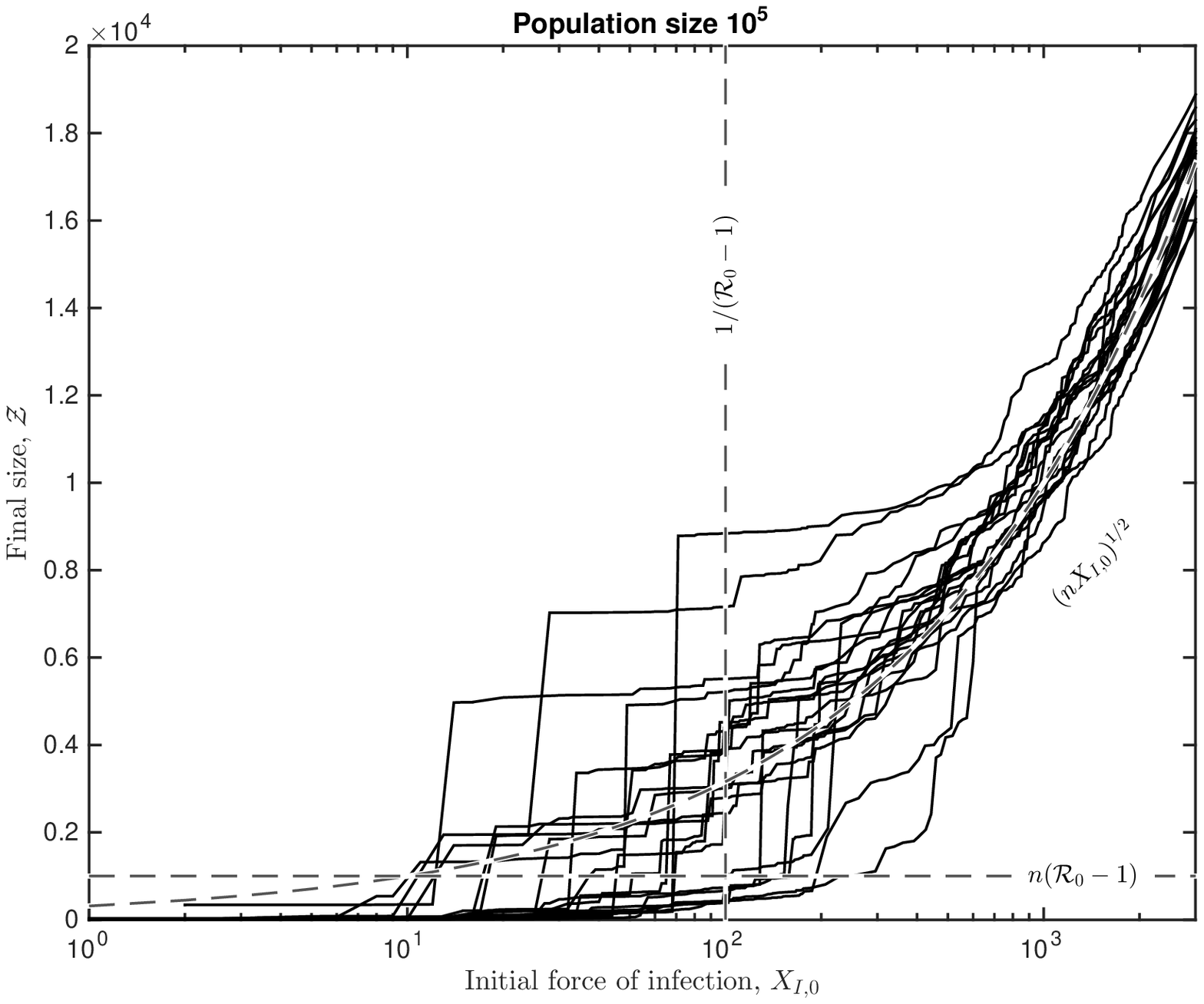}
\\[5mm]
\includegraphics[width = 0.45\textwidth]{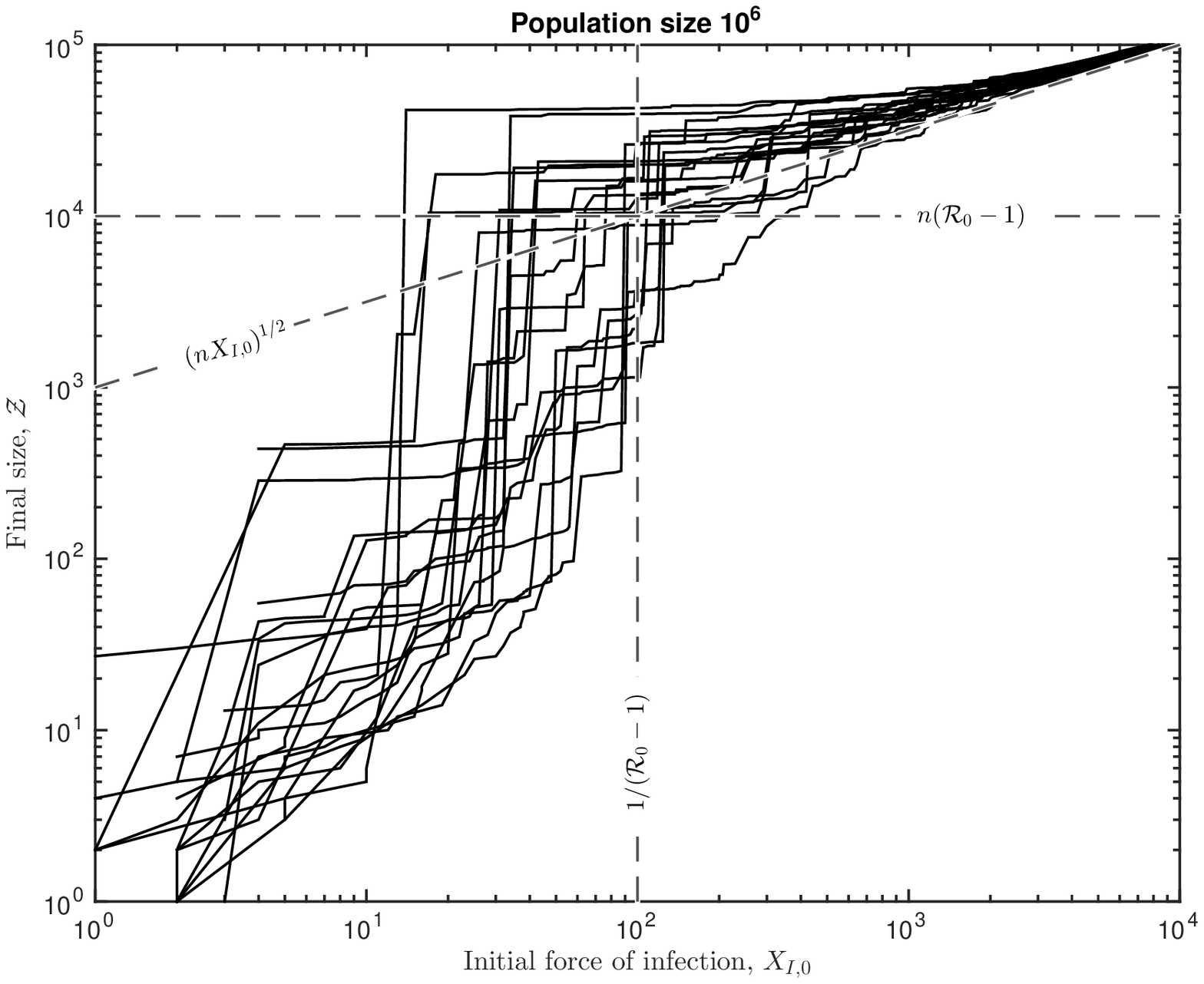}\quad
\includegraphics[width = 0.45\textwidth]{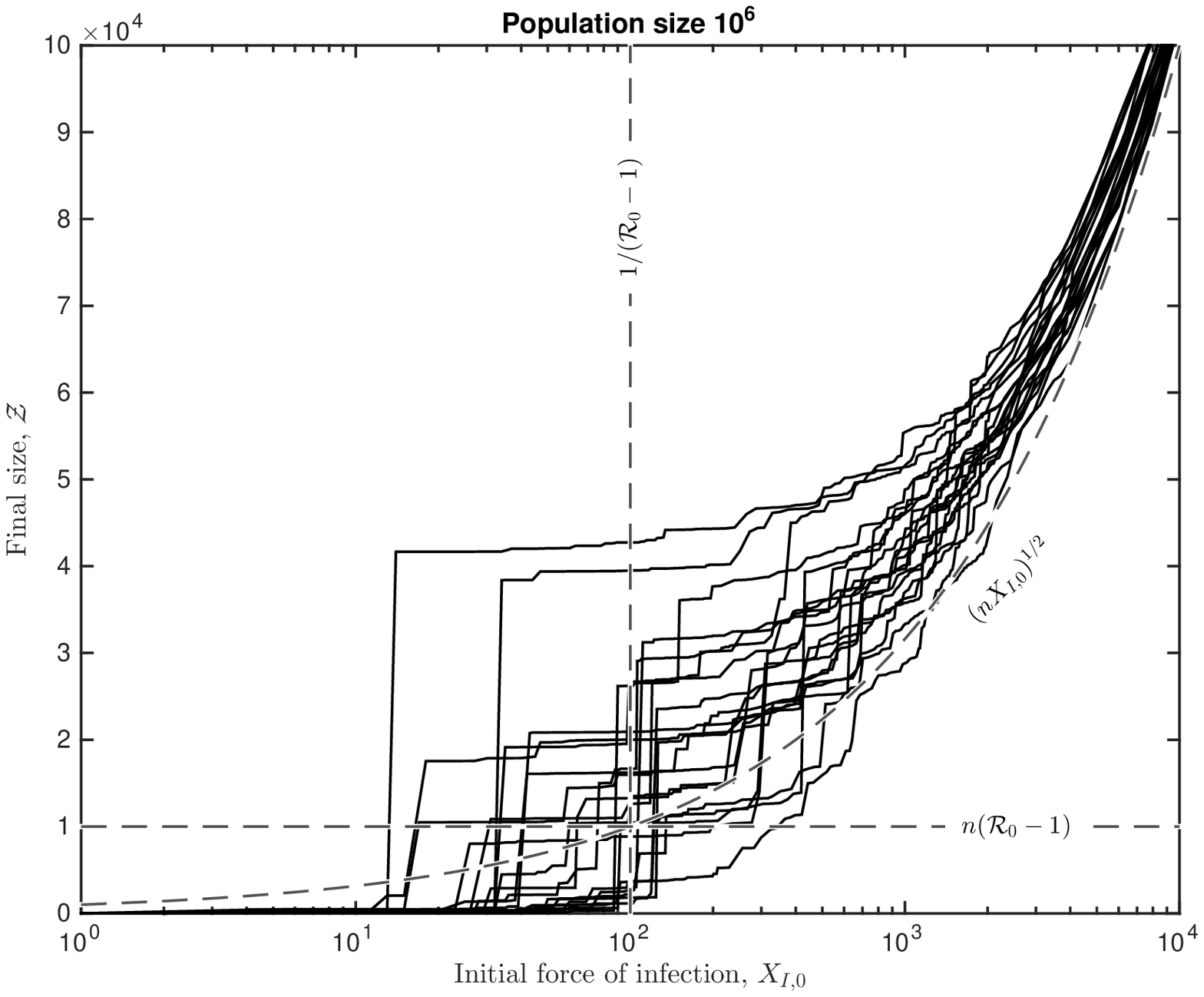}
\\[5mm]
\includegraphics[width = 0.45\textwidth]{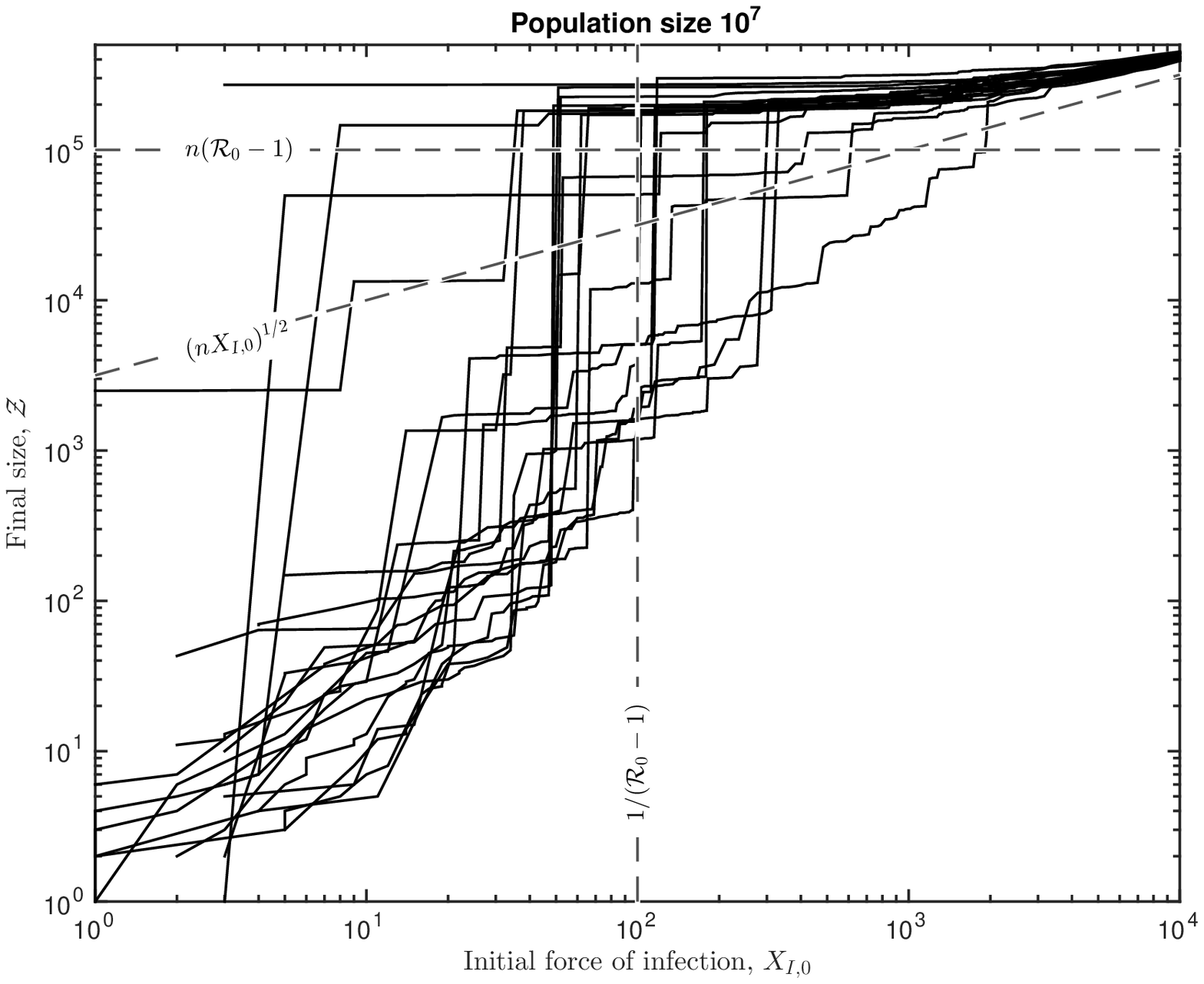}\quad
\includegraphics[width = 0.45\textwidth]{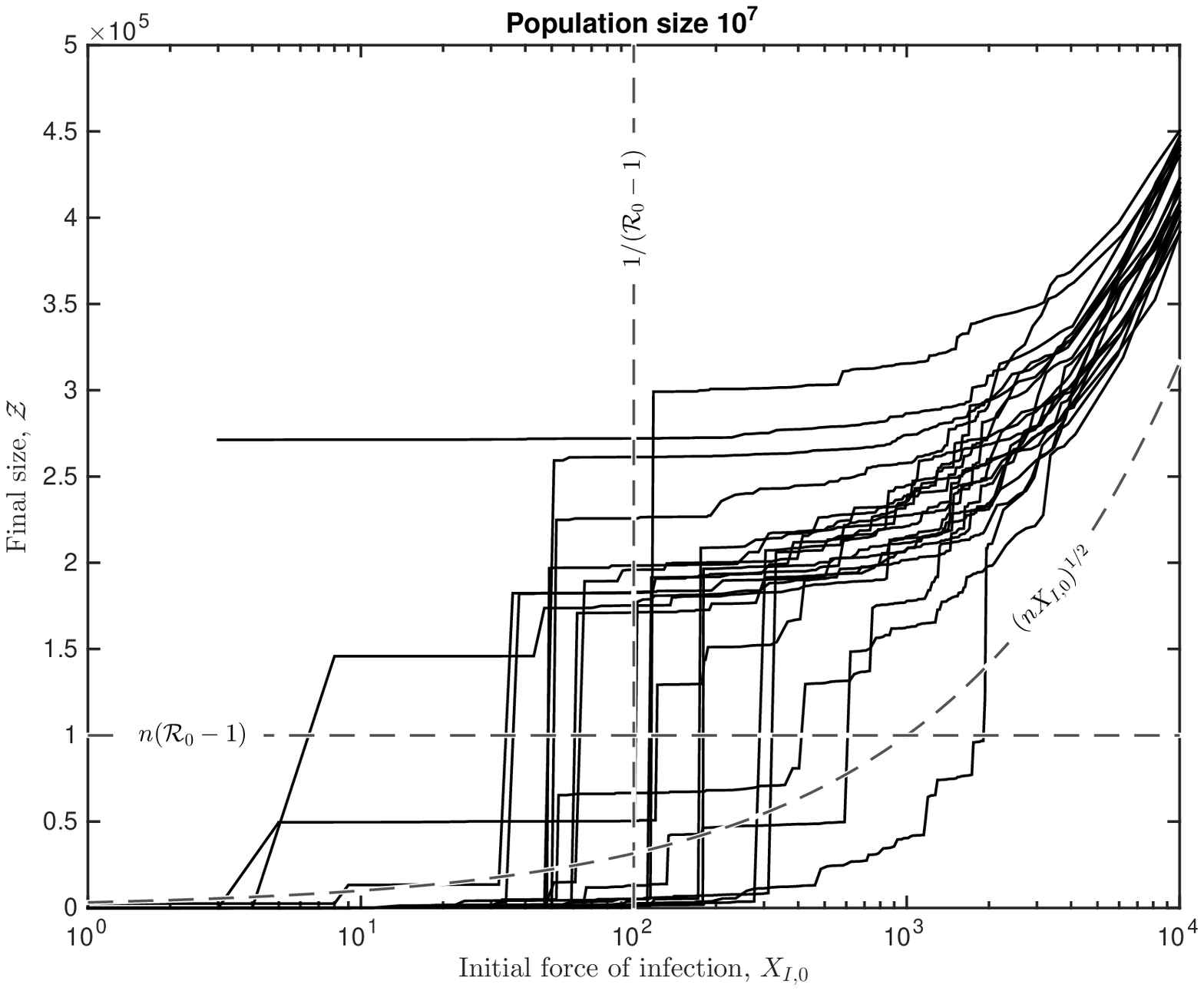}
\caption{The relationship between epidemic final size and initial force of
	infection for 20 realisations of the network Sellke construction for $n=10^5$
	(top), $n=10^6$ (middle) and $n=10^7$ (bottom), $\rho=1$, $\beta=1$, and
	Poisson degree distribution with mean $\lambda=2.02$ leading to
$\mathcal{R}_0 =1.01$.}
\label{fig:fs}
\end{figure}

\end{document}